\documentclass[11pt]{article}
\usepackage{amssymb}
\usepackage{graphicx}
\usepackage{xcolor} % A package to add color.
\usepackage{tensor}
\usepackage{fullpage} % Sets all margins to 1 in.  
\usepackage{amsmath}
\usepackage{amsthm}
\usepackage{verbatim}
\usepackage{hyperref}
\usepackage{enumitem}
\setlist[enumerate]{leftmargin=1.5em}
\setlist[itemize]{leftmargin=1.5em}

\providecommand{\bysame}{\leavevmode\hbox to3em{\hrulefill}\thinspace}
\providecommand{\MR}{\relax\ifhmode\unskip\space\fi MR }
\providecommand{\href}[2]{#2}

\usepackage{seqsplit} 
\definecolor{green}{rgb}{0,0.8,0} % Redefines the color green.
\newtheorem{maintheorem}{Theorem}

\newtheorem{theorem}{Theorem}[section]

\newtheorem{lemma}[theorem]{Lemma}
\newtheorem{proposition}[theorem]{Proposition}

\theoremstyle{definition}

\theoremstyle{remark}
\newtheorem{remark}[theorem]{Remark}

\numberwithin{equation}{section}
\newcommand{\nrm}[1]{\Vert#1\Vert}

\newcommand{\nnrm}[1]{{\vert\kern-0.25ex\vert\kern-0.25ex\vert #1 
    \vert\kern-0.25ex\vert\kern-0.25ex\vert}}

\newcommand{\supp}{{\mathrm{supp}}\,}

\newcommand{\lap}{\Delta}

\newcommand{\ud}{\mathrm{d}}
\newcommand{\rd}{\partial}
\newcommand{\nb}{\nabla}

\newcommand{\alp}{\alpha}
\newcommand{\bt}{\beta}
\newcommand{\gmm}{\gamma}

\newcommand{\dlt}{\delta}

\newcommand{\eps}{\epsilon}

\newcommand{\lmb}{\lambda}
\newcommand{\Lmb}{\Lambda}

\newcommand{\tht}{\theta}

\newcommand{\omg}{\omega}
\newcommand{\Omg}{\Omega}

\newcommand{\bbR}{\mathbb R}

\newcommand{\bbT}{\mathbb T}

\newcommand{\bbZ}{\mathbb Z}

\begin{document}

\title{Strong illposedness for SQG in critical Sobolev spaces}%: Title of the article
\author{In-Jee Jeong\thanks{Department of Mathematics and RIM, Seoul National University. E-mail: injee\_j@snu.ac.kr} \and Junha Kim\thanks{Department of Mathematics, Chung-ang University. E-mail: jha02@cau.ac.kr}} 
\date{\today}

%\thanks{}%
%\subjclass{}%
%\keywords{}%

%\date{\today}%
%\dedicatory{}%
%\commby{}%
% ----------------------------------------------------------------

\maketitle

% ----------------------------------------------------------------

\begin{abstract}
	We prove that the inviscid surface quasi-geostrophic (SQG) equations are strongly ill-posed in critical Sobolev spaces: there exists an initial data $H^{2}(\bbT^2)$ without any solutions in $L^\infty_{t}H^{2}$. Moreover, we prove strong critical norm inflation for $C^\infty$--smooth data. Our proof is robust and extends to give similar ill-posedness results for the family of modified SQG equations which interpolate the SQG with two-dimensional incompressible Euler equations. 
\end{abstract}

%\tableofcontents

\section{Introduction}

\subsection{Main results}

In this paper, we are concerned with the Cauchy problem for the inviscid surface quasi-geostrophic (SQG) equations on $\bbT^2 = (\bbR/\bbZ)^2$
\begin{equation}  \label{eq:SQG} \tag{SQG}
\left\{
\begin{aligned} 
\rd_t\tht + u\cdot \nb \tht = 0, \\
u= \nb^\perp (-\lap)^{-\frac{1}{2}} \tht.
\end{aligned}
\right.
\end{equation} Our first main result shows that \textit{strong norm inflation} occurs for the solution map of \eqref{eq:SQG} in $H^2(\bbT^2)$ with $C^{\infty}$--smooth solutions.
\begin{maintheorem}[Strong norm inflation]\label{thm:main}
	For any $\eps, \dlt, A> 0$, there exists $\tht_{0} \in  C^\infty(\bbT^2) $ satisfying $\nrm{\tht_0}_{H^2 \cap W^{1,\infty}} < \eps$ such that the unique local-in-time smooth solution $\tht$ to \eqref{eq:SQG} with initial data $\tht_{0}$ exists on $[0,\dlt^*]$ for some $0<\dlt^*\le \dlt$ and satisfies 
		$$	\sup_{t\in[0,\dlt^*]} \nrm{\tht(t,\cdot)}_{H^2}>A . $$
\end{maintheorem} 

\noindent The above result implies that the solution operator defined from $H^2 \cap C^\infty$ to $H^2$ by $\tht_0 \mapsto \tht(t)$ for any $t>0$ cannot be continuous at the trivial solution. On the other hand, the following result shows that actually it is impossible to define the solution operator from $H^2$ to $L^\infty_tH^2$. 
\begin{maintheorem}[Nonexistence]\label{thm:nonexist}
	For any $\eps > 0$, there exists $\tht_{0} \in H^{2}\cap W^{1,\infty}(\bbT^2)$ satisfying \begin{equation*}
		\begin{split}
			\nrm{\tht_0}_{H^{2} \cap W^{1,\infty}} < \eps 
		\end{split}
	\end{equation*} such that there is \emph{no} solution to \eqref{eq:SQG} with initial data $\tht_{0}$ belonging to $L^\infty([0,\dlt];H^{2}(\bbT^{2}))$ with any $\dlt>0$. 
\end{maintheorem}

\begin{remark}
	We give a few remarks relevant to the statements above. \begin{itemize}
		\item With a rather straightforward modification of our proof, the space $H^{2}$ in Theorems \ref{thm:main} and \ref{thm:nonexist} can be replaced with $W^{1+\frac{2}{p},p}$ with \textit{any} $p>1$. Later, we shall sketch the proof in the endpoint case $p=\infty$. Moreover, the domain $\bbT^2$ can be replaced with $\bbR^2$ or bounded domains having symmetry axes.
		\item The initial data for which non-existence occur can be given explicitly; see \eqref{eq:nonexist-data} below. 
		\item The arguments we present can be adapted to prove ill-posedness for the case of \textit{modified} (and \textit{logarithmically regularized}) SQG equations; see Subsection \ref{subsec:mSQG} below. 
	\end{itemize}
\end{remark}

\subsection{Well-posedness theory for SQG}

To put the above ill-posedness results into context, let us briefly recall the well-posedness theory for the SQG equation. Depending on the regularity of the solutions considered, one has the following categories: \begin{itemize}
	\item \textbf{Strong solutions: local existence and uniqueness}. Using the Kato--Ponce commutator estimate, one obtains the following a priori estimate \begin{equation*}
		\begin{split}
			\frac{d}{dt} \nrm{\tht}_{H^s} \le C\nrm{\nb u }_{L^{\infty}} \nrm{\tht}_{H^s}
		\end{split}
	\end{equation*} for a solution of \eqref{eq:SQG}, which allows one to close $\nrm{\tht(t)}_{H^s} \lesssim \nrm{\tht_0}_{H^s}$ for $t \lesssim \nrm{\tht_0}_{H^s}^{-1}$ once $s>2$, using that $\nrm{\nb u}_{L^\infty} \lesssim \nrm{\tht}_{H^s}$. Similarly, $H^s$ can be replaced with $W^{s,p}$, as long as $s>1+\frac{2}{p}$. Based on this a priori estimate, one can prove local existence and uniqueness of a strong solution in the class $L^\infty_{t} W^{s,p}$ with $s>1+\frac{2}{p}$. %although is not known whether such solutions could blow up in a finite time. 
	On the other hand, note that the borderline inequality $\nrm{\nb u}_{L^\infty} \lesssim \nrm{\tht}_{H^2}$ {fails}; this makes the Sobolev space $H^2$ (and similarly $W^{1+2/p,p}$) \textit{critical} for local well-posedness. This space is also \textit{scaling-critical}: the critical norm is left invariant under the transformation \begin{equation}\label{eq:scaling}
	\begin{split}
		\tht(t,x) \mapsto \lmb^{-1} \tht(t,\lmb x), \qquad u(t,x) \mapsto \lmb^{-1} u(t,\lmb x). 
	\end{split} 
\end{equation} While not much is known for long-time dynamics of \eqref{eq:SQG}, see a recent breakthrough of He--Kiselev \cite{HeKi} (and references therein) for a construction of smooth initial data with Sobolev norms growing at least exponentially for all times. Moreover, existence of traveling-wave solutions (\cite{Li-SQG,CQZZ-stab}) and rotating solutions (\cite{HH-sqg,HHH,CCG-duke}) are known.
	\item \textbf{Weak solutions: global existence}. Global existence of $L^p$--weak solutions is known, thanks to the works \cite{Res,Mar,BaeGra}. While such solutions are in general expected to be non-unique, see \cite{CCG-TRAN} for a uniqueness result for patches. On the other hand, for ``very'' weak solutions, non-uniqueness has been established--see \cite{BSV,CKL}. Note the gap of regularity between week and strong solutions. 
	\item \textbf{Ill-posedness in $W^{1,\infty}$}. To the best of knowledge, the only critical space ill-posedness result concerning (SQG) is the one given in \cite{EM1} for $W^{1,\infty},$ where a powerful general method for proving ill-posedness of active scalar systems in $L^\infty$--type spaces is developed. To be precise, the authors show that (\cite[Section 9.2]{EM1}) there exist smooth steady states $\bar{\tht}$ and a sequence of perturbations $\tilde{\tht}_0^{(\eps)}$ ($\eps\to 0^+$) so that the associated (SQG) solution $\tht^{(\eps)}$ with data $\bar{\tht} +\tilde{\tht}_0^{(\eps)}$ satisfies \begin{equation*}
		\begin{split}
			\nrm{\tht^{(\eps)}(0,\cdot) - \bar{\tht}}_{W^{1,\infty}} < \eps,\qquad \sup_{0<t<\eps}\nrm{ \tht^{(\eps)}(t,\cdot) - \bar{\tht} }_{W^{1,\infty}} > c
		\end{split}
	\end{equation*} where $c>0$ depends only on $\bar{\tht}$. It is very interesting to note that the authors use well-posedness in \textit{critical Besov spaces} with summability index 1. Such Besov well-posedness theory goes back to the pioneering work \cite{Vi1}. Our result (which applies in the $W^{1,\infty}$ case as well) basically says that one can take $\bar{\tht} \equiv 0 $ and replace $c$ by $\eps^{-1}$. On the other hand, one can restore well-posedness in $W^{1,\infty}$ by assuming some \textit{rotational symmetry} and anisotrophic H\"older regularity \cite{EJS}.
\end{itemize}
The current work settles the issue of \textit{strong ill-posedness} of (SQG) at critical Sobolev spaces, and we believe that this could be a first step in understanding the dynamics of ``slightly'' supercritical and subcritical solutions (e.g. evolution of $H^{s}$--data with $|s-2| \ll 1$), thereby bridging the gap between the theory of weak and strong solutions. Indeed, in a very recent work of Elgindi on singularity formation for the 3D Euler equations \cite{Elgindi-3D}, one of the key steps was to understand precisely the mechanism of $C^1$--illposedness. %For even more merits to study critical function spaces, see \cite{Jthesis}. 
Closing this section, let us mention some interesting works which seem contradictory to our main results: \begin{itemize}
	\item Miura \cite{Miura} proved that the fractionally dissipative SQG system \begin{equation}  \label{eq:SQG-diss}  
		\left\{
		\begin{aligned} 
			\rd_t\tht + u\cdot \nb \tht + (-\lap)^{\bt}\tht =0 , \\
			u= \nb^\perp (-\lap)^{-\frac{1}{2}} \tht,
		\end{aligned}
		\right.
	\end{equation} is actually \textit{well-posed} in the critical Sobolev space $H^{2-2\bt}$ for all $\bt>0$ (for data of any size), and this seems to suggest $H^2$ well-posedness of the inviscid system by taking $\bt\to0$! See \cite{Li,JKM1,JKM2} for related recent advances. 
	\item An invariant measure defined on $H^2(\bbT^2)$ which guarantees global well-posedness in $L^\infty_tH^2$ for any initial data in the support of the measure was constructed in \cite{FSy}. The data in Theorem \ref{thm:nonexist} certainly does not belong to the support of such a measure. 
\end{itemize} 

\subsection{Generalized SQG equations}\label{subsec:mSQG}

In the recent years, there has been significant interest in the study of so-called \textit{generalized} SQG equations, given by \begin{equation}  \label{eq:SQG-gen} 
	\left\{
	\begin{aligned} 
		\rd_t\tht + u\cdot \nb \tht = 0, \\
		u= \nb^\perp P(\Lmb) \tht,
	\end{aligned}
	\right.
\end{equation} where $P(\Lmb)$ is some Fourier multiplier, with $\Lmb = (-\lap)^{\frac12}$. Two distinguished cases are $P(\Lmb) = \Lmb^{-1}$ (SQG) and $P(\Lmb) = \Lmb^{-2}$ (2D incompressible Euler). Of particular interest is the case of $\alp$--SQG systems given by $P(\Lmb) = \Lmb^{-\alp}$ with $1 \le \alp \le2$, which interpolates the SQG and 2D Euler equations. The $L^2$--based critical Sobolev space is then given by $H^{3-\alp}$, and let us point out that the methods developed in the current work can handle the entire range $1 \le \alp \le 2$ without any essential change in the proof. One could consider more general symbols such as $P(\Lmb) = \Lmb^{-\alp} \log^{-\gmm}(10+\Lmb)$ with $\gmm>0$, which give rise to the so-called \textit{logarithmically regularized} systems (\cite{ChaeWu,CCW3,DongLi}). It is known that if the power of the logarithm is sufficiently large, then one can restore well-posedness in $H^{3-\alp}$ (\cite{ChaeWu}), but at this point it is more appropriate to regard a  logarithmically singularized Sobolev space to be critical. Indeed, one can see from our proof that there is a ``logarithmic'' room\footnote{strictly speaking, some power of the logarithm} in the arguments and therefore the same proof can cover same ill-posedness results in the slightly logarithmically regularized systems. We shall not dwell on this issue any further. 

\subsection{Critical space illposedness for Euler}\label{subsec:2D-Euler}

It should be emphasized that the strong Sobolev ill-posedness statements, Theorems \ref{thm:main} and \ref{thm:nonexist}, were first established in a groundbreaking work of Bourgain and Li \cite{BL1,BL3D}, for the case of 2D and 3D Euler equations, respectively. Further developments, including the current work, seem to have been strongly inspired by these papers. Recently, Kwon \cite{Kwon} settled the problem of strong ill-posedness in $H^1$ for logarithmically regularized (strictly speaking, powers of the log less than $1/2$) 2D Euler equations, nicely complementing previous $H^1$ well-posedness from \cite{ChaeWu}. On the other hand, much simpler proofs of $H^1$ ill-posedness for 2D Euler, which also shows continuous-in-time degeneration of the solution in Sobolev spaces, have appeared in \cite{EJ,JY2}. Some details of these simplified arguments will be given in the next section. 

\section{Ingredients of the proof}\label{sec:ingredients} 

The purpose of this section is to sketch the main ingredients of the proof. Several key ideas have already appeared in earlier works establishing ill-posedness in the Euler case; we briefly review those in Subsection \ref{subsec:Euler}. Additional difficulties arising in the (generalized) SQG case and new ideas are covered then in Section \ref{subsec:SQG}. 

\subsection{Strategy in the Euler case}\label{subsec:Euler}

In this section, let us give an overview of the ill-posedness proof in the 2D Euler case. We recall that in $\bbT^2$, the Euler equations are given by \begin{equation}  \label{eq:Euler} \tag{Euler}
	\left\{
	\begin{aligned} 
		\rd_t\omg + u\cdot \nb \omg = 0, \\
		u= \nb^\perp (-\lap)^{-1} \omg.
	\end{aligned}
	\right.
\end{equation} In terms of $\omg$, the critical $L^2$--based Sobolev space is $H^1(\bbT^2)$; indeed, $\omg\in H^1$ barely fails to guarantee $\nabla u \in L^\infty$, which is necessary to close the a priori estimate in $H^1$. 

\medskip

\noindent \textit{Choice of data for Euler}. As a starting point of discussion, we present an interesting identity observed by T. Elgindi: \begin{equation}\label{eq:Euler-iden}
\begin{split}
	\frac{d}{dt} ( \nrm{ \rd_2 \omg }_{L^2}^2 - \nrm{\rd_1\omg}_{L^2}^2 ) = \frac12\int_{\bbT^2} \rd_1u_1 ((\rd_2\omg)^2+(\rd_1\omg)^2) + \omg\rd_1\omg\rd_2\omg \, \ud x. 
\end{split}
\end{equation} For $\omg_0\in L^\infty$, Yudovich theory provides a unique global solution in $L^\infty([0,\infty)\times L^\infty$, and note that the last term in \eqref{eq:Euler-iden} cannot contribute to a large growth of the $H^1$--norm in a small time interval. Therefore, to prove existence of an $H^1\cap L^\infty$--initial data $\omg_0$ which ``escapes'' $H^1$ instantaneously, the goal would be to find $\omg_0 \in H^1\cap L^\infty$ such that \begin{equation}\label{eq:target}
\begin{split}
	\int_0^t \int_{\bbT^2} \rd_1u_1 (\rd_2\omg)^2 \, \ud x  = +\infty 
\end{split}
\end{equation} for \textit{any} $t>0$, where $\omg$ is the Yudovich solution with data $\omg_0$ and $u = \nb^\perp (-\lap)^{-1} \omg.$ In particular, it is necessary that at the initial time, we have \begin{equation}\label{eq:key-initial}
\begin{split}
	\int_{\bbT^2} \rd_1u_{0,1} (\rd_2\omg_0)^2 \, \ud x = +\infty. 
\end{split}
\end{equation} The choice in \cite{EJ} was \begin{equation}\label{eq:vol-ini}
\begin{split}
	\omg_0(x) \simeq \frac{x_1x_2}{|x|^2} |\ln|x||^{-\alp}, \qquad |x| \ll 1
\end{split}
\end{equation} since then (\cite{BC,Den2}) \begin{equation}\label{eq:vel-ini}
\begin{split}
	\rd_1u_{0,1}(x) \simeq |\ln|x||^{1-\alp}, \qquad |x| \ll 1
\end{split}
\end{equation} which in particular guarantees \eqref{eq:key-initial} for a range of $\alp>0$. 

\medskip

\noindent \textit{Hyperbolic flow}. Geometrically, vorticity which is positive on the first quadrant and odd with respect to both axes (as in \eqref{eq:vol-ini}) induces velocity which is stretching in the $x_1$-direction and contracting in the other, which leads to squeezing of the vorticity near the $x_1$-axis and growth of $H^1$--norm. This so-called ``hyperbolic flow scenario'' has been used to produce Euler solutions with gradient growth--see \cite{KS,Z,X,Den,Den2,Denisov-merging,EJSVP2}. Flattening of the vorticity level sets in such a flow configuration was studied in detail in \cite{Zlevel,Jeong}. 

\medskip

\noindent \textit{Regularization effect}. The main task is then to ensure that the velocity field, for a small time interval, retains its logarithmic divergence near the origin: indeed, instantaneous blow-up of the $H^1$--norm is not too difficult to see for the passive transport equation \begin{equation*}
	\begin{split}
		\rd_t \omg +u_0\cdot\nb\omg=0,
	\end{split}
\end{equation*} by solving the equation along the flow generated by $u_0$. When one tries to replace $u_0$ by $u$, a fundamental difficulty arises: anisotropic stretching of the vorticity regularizes the velocity. Indeed, rather involved computations in \cite{EJSVP1} suggests the asymptotics $\nrm{\nb u (t)}_{L^\infty}\lesssim t^{-1}$, which is barely non-integrable in time; this indicates that it could be a very delicate problem to verify \eqref{eq:target}. This upper bound of $t^{-1}$ can be seen for instance by solving the passive problem above and re-calculating the associated velocity at later times. 

\medskip

\noindent \textit{Key Lemma and Lagrangian approach}. Towards the goal of obtaining a \textit{lower bound} on the velocity gradient $|\nb u(t)|\gtrsim t^{-1}$, one needs to have a robust way of estimating the velocity gradient and prove some ``stability'' of the initial data. Regarding the former, the celebrated \textbf{Key Lemma} of Kiselev--Sverak asserts that (stated roughly) \begin{equation}\label{eq:keyLemma-KS}
	\begin{split}
		\frac{u_1(x)}{x_1} \simeq \int_{[x_1,1]\times[x_2,1]} \frac{y_1y_2}{|y|^4} \omg(y) \, \ud y, 
	\end{split}
\end{equation} for $\omg\in L^\infty$ with odd-odd symmetry. Note that $u_1=0$ for $x_1=0$ by symmetry, so that the left hand side is an approximation of $\rd_1 u_1$. The lower bound of the form \eqref{eq:keyLemma-KS} has proven to be extremely powerful in establishing growth of the vorticity (\cite{KS,Z,X,KRYZ,GaPa,Elgindi-3D,HeKi}). It is interesting to note that Bourgain--Li independently derived similar lower bounds in \cite{BL1}. Next, regarding the issue of showing stability of the data, the key observation is the hierarchy of vortex dynamics expressed in \eqref{eq:keyLemma-KS}: the vorticity around a point $x$ is being affected mainly by the vorticity supported in $|y|\ge 2|x|$. This suggests that the chunk of vorticity supported far away from the origin is more stable, thereby contributing to the right hand side of \eqref{eq:keyLemma-KS} for a longer time interval, to squeeze the vorticity closer to the origin. The proof of such stability and squeezing phenomena should be done in the Lagrangian variable, using the transport formulas \begin{equation*}
\begin{split}
	\omg(t,x) = \omg_0(\Phi_t^{-1}x), \qquad \nb\omg(t,x) = \nb\omg_0(\Phi_t^{-1}(x)) \nb\Phi_t^{-1}(x) 
\end{split}
\end{equation*} where $\Phi_t$ is the flow map at time $t$. In the actual ill-posedness proofs, Lagrangian versions of the formula \eqref{eq:Euler-iden} are used. 

\subsection{Difficulties in the SQG case}\label{subsec:SQG}

Overall, the strategy of the ill-posedness proof in the SQG case is similar to that explained in the above for 2D Euler. Roughly speaking, the initial data is now modified to be \begin{equation*}
	\begin{split}
		\tht_0 \simeq \frac{x_1x_2}{|x|} |\ln|x||^{-\alp}, \qquad |x| \ll 1, 
	\end{split}
\end{equation*} which is odd-odd and nonnegative in the first quadrant. The associated SQG velocity then satisfies the asymptotics \eqref{eq:vel-ini} with strong hyperbolicity near the origin, which should stretch $\tht$ near the $x_1$-axis. The issue is whether such a stretching effect is sufficiently strong to remove $\tht$ from the critical Sobolev space it started from. Let us now explain some main differences with the Euler case and new ideas employed to handle those. 

While the equation for $\tht$ in (SQG) is simply the transport equation exactly as in the 2D Euler case, probably the most significant difference is that while the $L^\infty$--norm is the common strongest conservation law, it is critical for 2D Euler but one order weaker for SQG. Furthermore, there is global well-posedness for 2D Euler with $\omg_0\in L^\infty$ (Yudovich \cite{Y1}), and the associated sharp estimates given by Yudovich theory have been very useful in understanding the dynamics.\footnote{Even in the 3D Euler case, Bourgain--Li \cite{BL3D} actually carefully identifies a class of initial data for which $\omg\in L^\infty$ propagates locally in time. Then, one can prove and utilize estimates similar to Yudovich's in 3D.} On the other hand, the corresponding quantity in the SQG case, $\nrm{\nb\tht}_{L^\infty}$, blows up together with $\nrm{\tht}_{H^2}$. 

It seems that the only way to handle this issue is to rely entirely on a contradiction argument--we \textit{assume} that there is an $L^\infty([0,T];H^2)$--solution, and then prove that for any $t>0$, the $H^2$--norm of the solution must be actually infinite. The whole point in this contradiction argument is that we can use the hypothetical $H^2$--bound to control the solution, an idea originated in \cite{BL1}. Again, the difficulty in the SQG case is that this hypothetical $H^2$ control is the only useful bound, whereas in the Euler case one has both $H^1$ and $L^\infty$ control. Fortunately, it turns out that having an $H^2$--bound guarantees that the velocity is log-Lipschitz, which implies in particular \textit{uniqueness} in the class $L^\infty_tH^2$ (this guarantees propagation in time of odd-odd symmetry and non-negativity) and existence of the flow map. That is, an $L^\infty_tH^2$--solution is Lagrangian, and therefore we can apply transport formulas to understand the dynamics. 

Under the contradiction hypothesis, the main part of the argument is to derive and apply a version of the Key Lemma adapted to the SQG case. Series technical difficulties appear; to begin with, in the remainder estimate of the Key Lemma (see estimates \eqref{u1} and \eqref{u2} in Lemma \ref{key_lem}) we are only allowed to use $\tht\in H^2$. As a consequence, the remainder term blows up super-logarithmically (the power 3/2 in \eqref{u2}) as the point $x$ approaches the axes, whereas only logarithmic errors are allowed in the nonexistence proof. It seems that the only way to overcome this issue is to track carefully the geometry of the support of $\tht$ in time so that the problematic remainder term disappears. To achieve this, we replace $\tht_0$ with a disjoint union of dyadic ``bubbles'' satisfying the same asymptotics as $|x|\to0$ (see \eqref{eq:nonexist-data} below) and obtain detailed information on the location of these bubbles for an interval of time inductively, starting from the largest one. Such refined information appear in technical \textbf{Claims I, II, III} in the proof. In the context of controlling bubbles, another significant difference with Euler is that the ``self-interaction'' of a bubble is \textit{not} a bounded term anymore. To overcome this issue we need to track the location of the ``top point'' of each bubble, which is the slowest point but does not suffer from self-interactions. 

Closing this section, we remark that the versions of the Key Lemma derived in this work should be useful in improving previous growth results for the active scalar equations, as we handle the remainder term only with the critical quantity.

\subsection{Organization of the paper} 

The rest of this paper is organized as follows.  The main technical tool, which we shall refer to as Key Lemma, is stated and proved in Section \ref{sec:key}. After that, the proofs of Theorems \ref{thm:nonexist} and \ref{thm:main} are given in Sections \ref{sec:nonexist} and \ref{sec:norm-inflation}.  
 
\section{The Key Lemma}\label{sec:key}

To begin with, we recall the famous Hardy's inequality. 
\begin{lemma}[Hardy's inequality]\label{lem_hd}
	Let $f$ be a smooth function defined on the interval $(0,l)$ with $f(0) = 0$. Then we have
	\begin{equation*}
		\begin{split}
			\nrm{x^{-1}f(x)}_{L^2(0,l)} \leq 2 \nrm{\partial f(x)}_{L^2(0,l)}
		\end{split}
	\end{equation*}
	and
	\begin{equation*}
		\begin{split}
			\nrm{x^{-2}f(x)}_{L^2(0,l)}^2 \leq 2 \nrm{\partial^2 f(x)}_{L^2(0,l)}^2
		\end{split}
	\end{equation*}
	for all $l \in [0,1]$.
\end{lemma}
\begin{proof}
	By the fundamental theorem of calculus and the assumption for $f$, we see for a.e. $x_2 \in [0,z_2]$ that
	\begin{align*}
		\int_0^l \frac {f(x)^2}{x^2}\, \mathrm{d}x &= - \frac 1l f(l)^2 + 2\int_0^l \frac {f(x)}{x} \partial f(x) \,\mathrm{d}x \\
		&\leq 2 \int_0^l \frac {f(x)}{x} \partial f(x) \mathrm{d}x.
	\end{align*}
	Using H\"{o}lder's inequality and Young's inequality, we have
	$$
	\int_0^l \frac {f(x)^2}{x^2}\, \mathrm{d}x \leq 4 \int_0^l |\partial f(x)|^2 \,\mathrm{d}x.
	$$
	Similarly, we have for a.e. $x_2 \in [0,z_2]$ that
	\begin{align*}
		\int_0^l \frac {f(x)^2}{x^4}\, \mathrm{d}x &= - \frac 1{3l^3} f(l)^2 + \frac 23 \int_0^l \frac {f(x)}{x^3} \partial f(x) \,\mathrm{d}x \leq \frac 23 \int_0^l \frac {f(x)}{x^2} \frac {\partial f(x)}{x} \mathrm{d}x \leq \frac 12 \int_0^l \frac {\partial f(x)^2}{x^2}\, \mathrm{d}x.
	\end{align*}
	Applying the above estimate, we obtain
	$$
	\int_0^l \frac {f(x)^2}{x^4}\, \mathrm{d}x \leq 2 \int_0^l |\partial^2 f(x)|^2 \,\mathrm{d}x.
	$$
	This completes the proof.
\end{proof}

We shall now state and prove the Key Lemma. For convenience, we shall normalize the SQG Biot--Savart law in a way that  
\begin{equation*}
	u(t,x) =  \sum_{n \in \mathbb{Z}^2} \int_{[-1,1)^2} \frac {(x -(y + 2n))^{\perp}}{|x - (y+2n)|^3} \theta(t,y) \,\mathrm{d}y
\end{equation*}
holds.
\begin{lemma}\label{key_lem}
	We impose the following assumptions on $\tht \in H^{2}$: \begin{itemize}
		\item $\tht$ is odd with respect to both axes, i.e., $\theta(x) = - \theta(\bar{x}) = \theta(-x) = -\theta(\tilde{x})$ where $\bar{x} := (x_1,-x_2)$ and $\tilde{x} := (-x_1,x_2)$;
		\item $\tht$ vanishes near the axis; to be precise, for any $x \ne (0,0)$ satisfying either $x_1 = 0$ or $x_2=0$, there exists an open neighborhood of $x$ such that $\tht$ vanishes. 
	\end{itemize} Then, for any $x$ satisfying $|x| < 1/4$ and $x_1 > x_2 > 0$, we have 
	\begin{equation}\label{u1}
		\left|\frac {u_1(x)}{x_1} - 12\int_{Q(x)} \frac { y_1 y_2 }{|y|^5} \theta(y) \,\mathrm{d}y\right| \le B_1(x)
	\end{equation}
	and
	\begin{equation}\label{u2}
		\left|\frac {u_2(x)}{x_2} + 12\int_{Q(x)} \frac { y_1 y_2 }{|y|^5} \theta(y) \,\mathrm{d}y  \right| \le \left( 1+\log \frac{x_1}{x_2} \right) B_2(x) + \left( 1+\log \frac{x_1}{x_2} \right)^{\frac 32} B_3(x),
	\end{equation}
	where $Q(x) := [2x_1,1] \times [0,1]$ and $B_1, B_2, B_3$ satisfy
	\begin{equation*}
		|B_1(x)| + |B_2(x)| \leq  C(\| \nabla^2 \theta \|_{L^2([0,1]^2)} + \| \theta \|_{L^\infty([0,1]^2)})
	\end{equation*}
	and
	\begin{equation*}|B_3(x)| \leq  C (\| \nabla \theta \|_{L^2(R(x))} + \| y_2^{-1} \partial_1 \theta(y) \|_{L^2(R(x))}), \qquad R(x) := [x_1/2,2x_1]\times [2x_2,1].\end{equation*}
\end{lemma}

\begin{remark}
	We clearly have that
	\begin{equation*}
		\| y_2^{-1} \partial_1 \theta(y) \|_{L^2(R(x))} \leq 2\| \nabla^2 \theta \|_{L^2([0,1]^2)}.
	\end{equation*}
\end{remark}

\begin{proof}
	We fix a point $x=(x_1,x_2)$ satisfying the assumptions of the lemma. After a symmetrization, we have 
	\begin{multline}\label{u_est}
		u(x) =  \sum_{n \in \mathbb{Z}^2} \int_{[0,1]^2} \bigg( \frac {(x -(y + 2n))^{\perp}}{|x - (y+2n)|^3} - \frac {(x -(\tilde{y} + 2n))^{\perp}}{|x - (\tilde{y}+2n)|^3} \\
		+ \frac {(x -(-y + 2n))^{\perp}}{|x - (-y+2n)|^3} - \frac {(x -(\bar{y} + 2n))^{\perp}}{|x - (\bar{y}+2n)|^3} \bigg) \theta(y) \,\mathrm{d}y .
	\end{multline}
	We estimate $u_1$ first. Consider
	\begin{equation*}
		\begin{aligned}
			I_1(n) &:= -\int_{[0,1]^2} \left( \frac {x_2 -(y_2 + 2n_2)}{|x - (y+2n)|^3} - \frac {x_2 -(y_2 + 2n_2)}{|x - (\tilde{y}+2n)|^3} \right) \theta(y) \,\mathrm{d}y, \\
			I_2(n) &:= -\int_{[0,1]^2} \left( \frac {x_2 -(-y_2 + 2n_2)}{|x - (-y+2n)|^3} - \frac {x_2 -(-y_2 + 2n_2)}{|x - (\bar{y}+2n)|^3} \right) \theta(y) \,\mathrm{d}y 
		\end{aligned}
	\end{equation*} so that from \eqref{u_est}
	\begin{equation*}
		u_1(x) =  \sum_{n \in \mathbb{Z}^2} (I_1(n) + I_2(n)).
	\end{equation*}
	We think of the cases $n = 0$ and $n \neq 0$ separately. For $n \neq 0$, we see that
	\begin{equation*}
		|I_1(n) + I_1(\tilde{n})| \leq O(|n|^{-4}) \| \theta \|_{L^\infty([0,1]^2)} x_1
	\end{equation*}
	and
	\begin{equation*}
		|I_2(n) + I_2(\bar{n})| \leq O(|n|^{-4}) \| \theta \|_{L^\infty([0,1]^2)} x_1.
	\end{equation*}
	Therefore, 
	\begin{equation}\label{far_est}
		\Bigg| \sum_{n \in \mathbb{Z}^2 \setminus \{0\}}(I_1(n) + I_2(n)) \Bigg| \leq C x_1 \| \theta \|_{L^\infty([0,1]^2)}.
	\end{equation}
	We now estimate the case of $n = 0$. Using
	\begin{equation}\label{AB_est}
		\frac 1{A^3} - \frac 1{B^3} = \frac {(B^2 - A^2)(A^2 + AB + B^2)}{A^3B^3 (A+B)},
	\end{equation}
	we have
	\begin{equation*}
		I_1(0) = -4x_1 \int_{[0,1]^2} \frac { y_1(x_2 - y_2) (|x - y|^2 + |x - y| |x - \tilde{y}| + |x - \tilde{y}|^2)}{|x - y|^3 |x - \tilde{y}|^3 (|x - y| + |x - \tilde{y}|)} \theta(y) \,\mathrm{d}y.
	\end{equation*}
	Noting that $[0,1]^2 = Q(x) \cup [0,2x_1] \times [2x_1,1] \cup [0,2x_1] \times [0,2x_1]$, we estimate the integral for each set. 
	
	\medskip
	
	\noindent \textit{(i) $Q(x)$} 
	
	\medskip
	
	\noindent In the case of $y \in Q(x)$, we can show that \begin{equation}\label{eq:equiv-y}
		\begin{split}
			\frac14 |y| \le |x-y| \le |y|, \qquad \frac 12|y| \le |x-\tilde{y}| \le 2|y|
		\end{split}
	\end{equation} because the first inequality comes from
	\begin{equation*}
		|x-y|^2 \geq |x_1-y_1|^2 \geq \frac 14 y_1^2 \geq \frac 18 |y|^2, \qquad y_1 \geq y_2
	\end{equation*} and
	\begin{equation*}
		|x-y|^2 = |x_1-y_1|^2 + |x_2 - y_2|^2 \geq \frac 14 y_1^2 + \frac 14 y_2^2, \qquad y_1 \leq y_2.
	\end{equation*} 
	%and the second inequality is obtained by	\begin{equation*}
	%	|x-\tilde{y}|^2 = |x_1+y_1|^2 + |x_2 - y_2|^2 \geq x_1^2 + y_1^2 + x_2^2 + y_2^2 - \frac 12 y_2^2 - 2x_2^2 \geq y_1^2 + \frac 12 y_2^2.
	%	\end{equation*} 
	The goal is to prove that \begin{equation*}
		\begin{split}
			-\int_{Q(x)} \frac { y_1(x_2 - y_2) (|x - y|^2 + |x - y| |x - \tilde{y}| + |x - \tilde{y}|^2)}{|x - y|^3 |x - \tilde{y}|^3 (|x - y| + |x - \tilde{y}|)} \theta(y) \,\mathrm{d}y  =:J
		\end{split}
	\end{equation*} satisfies \begin{equation}\label{eq:J-goal}
		\begin{split}
			\left| J - \frac{3}{2} \int_{Q(x)} \frac{y_1y_2}{|y|^5}\theta(y) \,\mathrm{d}y \right| \le C \Vert \nabla^2 \theta \Vert_{L^2}.
		\end{split}
	\end{equation} We separate 
	\begin{equation*}
		J = J_1 + J_2
	\end{equation*} where \begin{equation*}
		\begin{split}
			J_1 := \int_{Q(x)} \frac { y_1y_2  (|x - y|^2 + |x - y| |x - \tilde{y}| + |x - \tilde{y}|^2)}{|x - y|^3 |x - \tilde{y}|^3 (|x - y| + |x - \tilde{y}|)} \theta(y) \mathrm{d}y \\
		\end{split}
	\end{equation*} and \begin{equation*}
		\begin{split}
			J_2 := -\int_{Q(x)} \frac { y_1x_2  (|x - y|^2 + |x - y| |x - \tilde{y}| + |x - \tilde{y}|^2)}{|x - y|^3 |x - \tilde{y}|^3 (|x - y| + |x - \tilde{y}|)} \theta(y) \mathrm{d}y .
		\end{split}
	\end{equation*} Using \eqref{eq:equiv-y}, we may estimate \begin{equation*}
		\begin{split}
			|J_2| \le C|x| \int_{Q(x)} \frac 1{|y|^2} \frac{|\theta(y)|}{|y|^2} \,\mathrm{d}y .
		\end{split}
	\end{equation*} 
	Note that by H\"{o}lder's inequality,
	\begin{equation*}
		|x|\int_{Q(x)}  \frac{1}{|y|^2} \frac{|\theta(y)|}{|y|^2}\,\mathrm{d}y \le |x| \left( \int_{2x_1}^\infty \frac 1{r^3} \, \mathrm{d}r \right)^{\frac 12}\Vert |y|^{-2}\theta(y) \Vert_{L^2(Q(x))} \le C\Vert |y|^{-2}\theta(y) \Vert_{L^2([0,1]^2)}. 
	\end{equation*} Then with the Hardy's inequality we have
	\begin{equation*}
		|J_2| \leq C \| \nabla^2 \theta \|_{L^2([0,1]^2)}.
	\end{equation*}
	On the other hand, regarding $J_1$, we shall show that \begin{equation*}
		\begin{split}
			\left| J_1 - \frac{3}{2} \int_{Q(x)} \frac{y_1y_2}{|y|^5}\theta(y) \,\mathrm{d}y \right| \le C \| \nabla^2 \theta \|_{L^2([0,1]^2)}.
		\end{split}
	\end{equation*} For this purpose we simply write $J_1=J_{11}+J_{12}+J_{13}$ where \begin{equation*}
		\begin{split}
			J_{11} &= \int_{Q(x)} \frac { y_1y_2  |x - y|^2  }{|x - y|^3 |x - \tilde{y}|^3 (|x - y| + |x - \tilde{y}|)} \theta(y)\, \mathrm{d}y ,\\
			J_{12} &= \int_{Q(x)} \frac { y_1y_2  |x - y| |x - \tilde{y}| }{|x - y|^3 |x - \tilde{y}|^3 (|x - y| + |x - \tilde{y}|)} \theta(y)\, \mathrm{d}y ,\\
			J_{13} & = \int_{Q(x)} \frac { y_1y_2   |x - \tilde{y}|^2}{|x - y|^3 |x - \tilde{y}|^3 (|x - y| + |x - \tilde{y}|)} \theta(y)\, \mathrm{d}y 
		\end{split}
	\end{equation*}  and show that \begin{equation*}\label{eq:J_1k}
		\begin{split}
			\left| J_{1k} - \frac{1}{2} \int_{Q(x)} \frac{y_1y_2}{|y|^5}\theta(y) \,\mathrm{d}y \right| \le C \| \nabla^2 \theta \|_{L^2([0,1]^2)}
		\end{split}
	\end{equation*} for each $k = 1, 2, 3$. We supply the proof only for the case $k = 1$, since the others can be treated similarly. We directly compute \begin{equation*}
		\begin{split}
			J_{11} - \frac{1}{2} \int_{Q(x)} \frac{y_1y_2}{|y|^5}\theta(y) \,\mathrm{d}y & =  \int_{Q(x)} y_1y_2 \frac{2|y|^5 - |x-y||x-\tilde{y}|^3(|x-y|+|x-\tilde{y}|)}{2|y|^5|x-y||x-\tilde{y}|^3(|x-y|+|x-\tilde{y}|)}\theta(y) \,\mathrm{d}y.
		\end{split}
	\end{equation*} We rewrite the numerator as \begin{gather*}
		2|y|^5 - |x-y||x-\tilde{y}|^3(|x-y|+|x-\tilde{y}|) \\ = (|y|^2 - |x-y|^2)|y|^3 + |x-y|^2(|y|^3 - |x-\tilde{y}|^3)  + (|y|-|x-y|)|y|^4 + |x-y| (|y|^4 - |x-\tilde{y}|^4) 
	\end{gather*} and further rewriting \begin{equation*}
		|y|-|x-y| = \frac{|y|^2 - |x-y|^2}{|y|+|x-y|}, \qquad |y|^3-|x-\tilde{y}|^3 = \frac{(|y|^2-|x-\tilde{y}|^2)(|y|^2+ |y||x-\tilde{y}|+|x-\tilde{y}|^2)}{|y|+|x-\tilde{y}|},
	\end{equation*} we see using \eqref{eq:equiv-y} that \begin{equation*}
		\left|2|y|^5 - |x-y||x-\tilde{y}|^3(|x-y|+|x-\tilde{y}|)\right| \le C|x||y|^4.
	\end{equation*} Then, we can infer that \begin{equation*}
		\left|J_{11} - \frac{1}{2} \int_{Q(x)} \frac{y_1y_2}{|y|^5}\theta(y) \,\mathrm{d}y \right| \le C|x|\int_{Q(x)}  \frac{1}{|y|^2} \frac{|\theta(y)|}{|y|^2}\,\mathrm{d}y \leq C \| \nabla^2 \theta \|_{L^2([0,1]^2)}.
	\end{equation*} 
	Collecting the estimates for $J_1$ and $J_2$ gives \eqref{eq:J-goal}.
	
	\medskip
	
	\noindent \textit{(ii) $[0,2x_1] \times [2x_1,1]$} 
	
	\medskip
	
	\noindent In the case of $y \in [0,2x_1] \times [2x_1,1]$, using $y_1 \leq y_2$, we can see that
	\begin{equation*}
		\frac 12{y_2} \leq |x-y| \leq 2y_2, \qquad \frac 12{y_2} \leq |x-\tilde{y}| \leq 2y_2.
	\end{equation*}
	Thus, H\"{o}lder's inequality and Hardy's inequality imply that
	\begin{equation}\label{est_ii}
		\begin{aligned}
			&\left|-\int_{[0,2x_1] \times [2x_1,1]} \frac { y_1(x_2 - y_2) (|x - y|^2 + |x - y| |x - \tilde{y}| + |x - \tilde{y}|^2)}{|x - y|^3 |x - \tilde{y}|^3 (|x - y| + |x - \tilde{y}|)} \theta(y) \,\mathrm{d}y \right| \\
			&\hphantom{\qquad\qquad\quad}\leq C \int_{[0,2x_1] \times [2x_1,1]} \frac {1}{y_2} \frac {\theta(y)}{y_2^2} \,\mathrm{d}y \\
			&\hphantom{\qquad\qquad\quad}\leq C \left( \int_{[0,2x_1] \times [2x_1,1]} \frac 1{y_2^2} \,\mathrm{d}y \right)^{\frac 12} \| y_2^{-2} \theta(y) \|_{L^2([0,1]^2)} \\
			&\hphantom{\qquad\qquad\quad}\leq C \| \nabla^2 \theta \|_{L^2([0,1]^2)} .
		\end{aligned}
	\end{equation}
	
	\medskip
	
	\noindent \textit{(iii) $[0,2x_1]^2$} 
	
	\medskip
	
	\noindent Due to $\theta(y_1,0) =0$, using integration by parts gives
	\begin{align*}
		&-\int_{[0,2x_1]^2} \frac { y_1(x_2 - y_2) (|x - y|^2 + |x - y| |x - \tilde{y}| + |x - \tilde{y}|^2)}{|x - y|^3 |x - \tilde{y}|^3 (|x - y| + |x - \tilde{y}|)} \theta(y) \,\mathrm{d}y \\
		&\hphantom{\qquad\qquad}= \frac 1{x_1} \int_{[0,2x_1]^2} \left( \frac 1{|x-y|} - \frac 1{|x-\tilde{y}|} \right) \partial_2 \theta(y)\, \mathrm{d}y \\
		&\hphantom{\qquad\qquad\qquad\qquad}- \frac 1{x_1} \int_0^{2x_1} \left( \frac 1{|(x_1-y_1,x_2-2x_1)|} - \frac 1{|(x_1+y_1,x_2-2x_1)|} \right) \theta(y_1,2x_1) \,\mathrm{d}y_1 .
	\end{align*}
	By H\"{o}lder's inequality we estimate the second integral as 
	\begin{equation*}
		\left| - \frac 1{x_1} \int_0^{2x_1} \left( \frac 1{|(x_1-y_1,x_2-2x_1)|} - \frac 1{|(x_1+y_1,x_2-2x_1)|} \right) \theta(y_1,2x_1) \,\mathrm{d}y_1 \right| \leq C x_1^{-1} \| \theta \|_{L^\infty([0,2x_1]^2)}.
	\end{equation*}
	We notice that since $\theta$ vanishes near the axis it follows
	\begin{equation}\label{tht_v}
		\begin{aligned}
			\| \theta \|_{L^\infty([0,2x_1]^2)} &\leq \sup_{y_1 \in [0,2x_1]} \int_0^{2x_1} |\partial_2 \theta(y_1,y_2)| \,\mathrm{d}y_2 \\
			&\leq (2x_1)^{\frac 12} \bigg\| \sup_{y_1 \in [0,2x_1]} |\partial_2\theta(y_1,\cdot)| \bigg\|_{L^2(0,2x_1)} \\
			&\leq 2x_1 \| \partial_1 \partial_2 \theta \|_{L^2([0,2x_1]^2)}.
		\end{aligned}
	\end{equation}
	Thus, we have
	\begin{equation*}
		\left| - \frac 1{x_1} \int_0^{2x_1} \left( \frac 1{|(x_1-y_1,x_2-2x_1)|} - \frac 1{|(x_1+y_1,x_2-2x_1)|} \right) \theta(y_1,2x_1) \,\mathrm{d}y_1 \right| \leq C \| \nabla^2 \theta \|_{L^2([0,1]^2)}.
	\end{equation*}
	Calculating the first integral with H\"{o}lder's inequality, we see that
	\begin{align*}
		\left| \frac 1{x_1} \int_{[0,2x_1]^2} \left( \frac 1{|x-y|} - \frac 1{|x-\tilde{y}|} \right) \partial_2 \theta(y)\, \mathrm{d}y \right| &\leq \frac 2{x_1} \int_{[0,2x_1]^2} \frac {|\partial_2 \theta(y)|}{|x-y|}\, \mathrm{d}y \\
		&\leq \frac 2{x_1} \left( \int_0^{4x_1} r^{-\frac 13} \,\mathrm{d}r \right)^{\frac 34} \| \partial_2 \theta \|_{L^4([0,2x_1]^2)} \\
		&\leq C x_1^{-\frac 12}  \| \nabla \theta \|_{L^4([0,2x_1]^2)}.
	\end{align*}
	Gagliardo-Nirenberg interpolation inequality implies
	\begin{equation*}
		x_1^{-\frac 12}  \| \nabla \theta \|_{L^4([0,2x_1]^2)} \leq C x_1^{-\frac 12} \| \nabla^2 \theta \|_{L^2([0,2x_1]^2)}^{\frac 34} \| \theta \|_{L^2([0,2x_1]^2)}^{\frac 14} + C x_1^{-2} \| \theta \|_{L^2([0,2x_1]^2)},
	\end{equation*} where the constant $C>0$ is independent of $x_1$. Applying Hardy's inequality to it, we have
	\begin{align*}
		x_1^{-\frac 12}  \| \partial_2 \theta \|_{L^4([0,2x_1]^2)} %&\leq C \| \nabla^2 \theta \|_{L^2([0,2x_1]^2)}^{\frac 34} \| y_1^{-2} \theta(y) \|_{L^2([0,2x_1]^2)}^{\frac 14} + C \| y_1^{-2} \theta(y) \|_{L^2([0,2x_1]^2)} \\
		&\leq C \| \nabla^2 \theta \|_{L^2([0,2x_1]^2)},
	\end{align*} and hence it follows that 
	\begin{equation*}
		\left| \frac 1{x_1} \int_{[0,2x_1]^2} \left( \frac 1{|x-y|} - \frac 1{|x-\tilde{y}|} \right) \partial_2 \theta(y)\, \mathrm{d}y \right| \leq C \| \nabla^2 \theta \|_{L^2([0,1]^2)}.
	\end{equation*}
	Combining the above estimates, we obtain that
	\begin{equation*}
		\left| -\int_{[0,2x_1]^2} \frac { y_1(x_2 - y_2) (|x - y|^2 + |x - y| |x - \tilde{y}| + |x - \tilde{y}|^2)}{|x - y|^3 |x - \tilde{y}|^3 (|x - y| + |x - \tilde{y}|)} \theta(y) \,\mathrm{d}y \right| \leq C \| \nabla^2 \theta \|_{L^2([0,1]^2)}.
	\end{equation*}
	\medskip
	We collect the estimates for each region and deduce that 
	\begin{align*}
		\left|\frac {I_1(0)}{x_1} - 6 \int_{Q(x)} \frac { y_1 y_2 }{|y|^5} \theta(y) \,\mathrm{d}y \right| \le C \| \nabla^2 \theta \|_{L^2([0,1]^2)}.
	\end{align*}
	We can estimate
	\begin{equation*}
		I_2(0) = 4x_1 \int_{[0,1]^2} \frac { y_1(x_2 + y_2) (|x + y|^2 + |x + y| |x - \bar{y}| + |x - \bar{y}|^2)}{|x + y|^3 |x - \bar{y}|^3 (|x - y| + |x - \bar{y}|)} \theta(y) \,\mathrm{d}y 
	\end{equation*}
	similarly with $I_1(0)$, resulting in the bound \begin{align*}
		\left|\frac {I_2(0)}{x_2} - 6 \int_{Q(x)} \frac { y_1 y_2 }{|y|^5} \theta(y) \,\mathrm{d}y \right| \le C \| \nabla^2 \theta \|_{L^2([0,1]^2)}.
	\end{align*} We omit the details. Thus  we have \eqref{u1}.
	
	Now we estimate $u_2$. Note that
	\begin{equation*}
		u_2(x) =  \sum_{n \in \mathbb{Z}^2} (I_3(n) + I_4(n)),
	\end{equation*}
	where
	\begin{equation*}
		\begin{aligned}
			I_3(n) &:= \int_{[0,1]^2} \left( \frac {x_1 -(y_1 + 2n_1)}{|x - (y+2n)|^3} - \frac {x_1 -(y_1 + 2n_1)}{|x - (\bar{y}+2n)|^3} \right) \theta(y) \,\mathrm{d}y, \\
			I_4(n) &:= \int_{[0,1]^2} \left( \frac {x_1 -(-y_1 + 2n_1)}{|x - (-y+2n)|^3} - \frac {x_1 -(-y_1 + 2n_1)}{|x - (\tilde{y}+2n)|^3} \right) \theta(y) \,\mathrm{d}y .
		\end{aligned}
	\end{equation*}
	Since we can similarly see that
	\begin{equation*}
		\Bigg| \sum_{n \in \mathbb{Z}^2 \setminus \{0\}}(I_3(n) + I_4(n)) \Bigg| \leq C x_2 \| \theta \|_{L^\infty([0,1]^2)},
	\end{equation*}
	it suffices to estimate for $n = 0$. Using \eqref{AB_est}, we have
	\begin{equation*}
		I_3(0) = 4x_2 \int_{[0,1]^2} \frac { y_2(x_1 - y_1) (|x - y|^2 + |x - y| |x - \bar{y}| + |x - \bar{y}|^2)}{|x - y|^3 |x - \bar{y}|^3 (|x - y| + |x - \bar{y}|)} \theta(y) \,\mathrm{d}y.
	\end{equation*}
	We divide the domain into four regions as $[0,1]^2 = Q(x) \cup [0,x_1/2] \times [0,2x_1] \cup [x_1/2,2x_1] \times [0,2x_1] \cup [0,2x_1] \times [2x_1,1]$ and estimate the integral in each region. 
	
	\medskip
	
	\noindent \textit{(i) $Q(x)$} 
	
	\medskip
	
	\noindent In the case of $y \in Q(x)$, recalling that \eqref{eq:equiv-y} holds, we can prove 
	\begin{equation*}
		\begin{gathered}
			\left| \int_{Q(x)} \frac { y_2(x_1 - y_1) (|x - y|^2 + |x - y| |x - \bar{y}| + |x - \bar{y}|^2)}{|x - y|^3 |x - \bar{y}|^3 (|x - y| + |x - \bar{y}|)} \theta(y) \,\mathrm{d}y + \frac{3}{2} \int_{Q(x)} \frac{y_1y_2}{|y|^5}\theta(y) \,\mathrm{d}y \right|\\
			\le C \Vert \nabla^2 \theta \Vert_{L^2([0,1]^2)}.
		\end{gathered}
	\end{equation*} Since the proof is parallel to the one for \eqref{eq:J-goal}, we omit it.
	
	\medskip
	
	\noindent \textit{(ii) $[0,2x_1] \times [2x_1,1]$} 
	
	\medskip
	
	\noindent For $y \in [0,2x_1] \times [2x_1,1]$, it follow that
	\begin{equation*}
		\frac 12{y_2} \leq |x-y| \leq 2y_2, \qquad {y_2} \leq |x-\bar{y}| \leq 2y_2.
	\end{equation*}
	Hence, we can show
	\begin{equation*}
		\left|\int_{[0,2x_1] \times [2x_1,1]} \frac { y_2(x_1 - y_1) (|x - y|^2 + |x - y| |x - \bar{y}| + |x - \bar{y}|^2)}{|x - y|^3 |x - \bar{y}|^3 (|x - y| + |x - \bar{y}|)} \theta(y) \,\mathrm{d}y \right| \leq C \| \nabla^2 \theta \|_{L^2([0,1]^2)} 
	\end{equation*}
	in a similar way with \eqref{est_ii}.
	
	\medskip
	
	\noindent \textit{(iii) $[0,x_1/2] \times [0,2x_1]$} 
	
	\medskip
	
	\noindent Since $y \in [0,x_1/2] \times [0,2x_1]$ implies
	\begin{equation*}
		\frac 12 x_1 \leq |x-y| \leq 4x_1, \qquad \frac 12 x_1 \leq |x-\bar{y}| \leq 4x_1,
	\end{equation*}
	with $y_2 \leq 2x_1$ we have
	\begin{align*}
		&\left|\int_{[0,x_1/2] \times [0,2x_1]} \frac { y_2(x_1 - y_1) (|x - y|^2 + |x - y| |x - \bar{y}| + |x - \bar{y}|^2)}{|x - y|^3 |x - \bar{y}|^3 (|x - y| + |x - \bar{y}|)} \theta(y) \,\mathrm{d}y \right| \\
		&\hphantom{\qquad\qquad\quad}\leq Cx_1^{-1} \int_{[0,x_1/2] \times [0,2x_1]} \frac {\theta(y)}{y_2^2} \,\mathrm{d}y \\
		&\hphantom{\qquad\qquad\quad}\leq C x_1^{-1} \left( \int_{[0,x_1/2] \times [0,2x_1]} 1 \,\mathrm{d}y \right)^{\frac 12} \| y_2^{-2} \theta(y) \|_{L^2([0,1]^2)} \\
		&\hphantom{\qquad\qquad\quad}\leq C \| \nabla^2 \theta \|_{L^2([0,1]^2)} .
	\end{align*}
	
	\medskip
	
	\noindent \textit{(iv) $[x_1/2,2x_1] \times [0,2x_1]$} 
	
	\medskip
	
	\noindent We claim that
	\begin{equation*}
		\int_{[x_1/2,2x_1] \times [0,2x_1]} \frac { y_2(x_1 - y_1) (|x - y|^2 + |x - y| |x - \bar{y}| + |x - \bar{y}|^2)}{|x - y|^3 |x - \bar{y}|^3 (|x - y| + |x - \bar{y}|)} \theta(y) \,\mathrm{d}y =: K
	\end{equation*} satisfies
	\begin{equation}\label{eq:K-goal}
		|K| \leq C \| \nabla^2 \theta \|_{L^2([0,1]^2)} \left( 1+\log \frac{x_1}{x_2} \right) + C (\| \nabla^2 \theta \|_{L^2(R(x))} + \| y_2^{-1} \partial_1 \theta(y) \|_{L^2(R(x))}) \left( 1+\log \frac{x_1}{x_2} \right)^{\frac 32} .
	\end{equation}
	Using integration by parts, we have
	\begin{equation*}
		K = K_1+K_2+K_3,
	\end{equation*}
	where
	\begin{align*}
		K_1 &:=- \frac 1{x_2} \int_{[x_1/2,2x_1] \times [0,2x_1]} \left( \frac 1{|x-y|} - \frac 1{|x-\bar{y}|} \right) \partial_1 \theta(y) \,\mathrm{d}y, \\
		K_2 &:= \frac 1{x_2} \int_0^{2x_1} \left( \frac 1{|(x_1,x_2-y_2)|} - \frac 1{|(x_1,x_2+y_2)|} \right) \theta(2x_1,y_2) \,\mathrm{d}y_2, \\
		K_3 &:= -\frac 1{x_2} \int_0^{2x_1} \left( \frac 1{|(x_1/2,x_2-y_2)|} - \frac 1{|(x_1/2,x_2+y_2)|} \right) \theta(x_1/2,y_2) \,\mathrm{d}y_2.
	\end{align*}
	With \eqref{tht_v} we may estimate $K_2$ as follows
	\begin{equation}\label{K2}
		\begin{aligned}
			|K_2|&= \left|\int_0^{2x_1} \frac{4y_2 \theta(2x_1,y_2)}{|(x_1,x_2-y_2)||(x_1,x_2+y_2)|(|(x_1,x_2-y_2)|+|(x_1,x_2+y_2)|)}  \,\mathrm{d}y_2 \right| \\
			&\leq C x_1^{-2} \int_0^{2x_1} \theta(2x_1,y_2)  \,\mathrm{d}y_2 \\
			&\leq C x_1^{-1} \| \theta \|_{L^\infty([0,2x_1]^2)} \\
			&\leq C \| \nabla^2 \theta \|_{L^2([0,1]^2)}.
		\end{aligned}
	\end{equation} And similarly, it is obtained
	\begin{equation*}
		|K_3|\leq C \| \nabla^2 \theta \|_{L^2([0,1]^2)}.
	\end{equation*}
	Noting that
	\begin{equation*}
		K_1= \int_{[x_1/2,2x_1] \times [0,2x_1]} \frac {-4 y_2 \partial_1 \theta(y)}{|x - y| |x - \bar{y}| (|x - y| + |x - \bar{y}|)} \,\mathrm{d}y,
	\end{equation*} we set
	\begin{equation*}
		K_1 = K_{11} + K_{12},
	\end{equation*} where
	\begin{align*}
		K_{11} &:= \int_{[x_1/2,2x_1] \times [0,2x_2]} \frac {-4 y_2 \partial_1 \theta(y)}{|x - y| |x - \bar{y}| (|x - y| + |x - \bar{y}|)} \,\mathrm{d}y, \\
		K_{12} &:= \int_{[x_1/2,2x_1] \times [2x_2,2x_1]} \frac {-4 y_2 \partial_1 \theta(y)}{|x - y| |x - \bar{y}| (|x - y| + |x - \bar{y}|)} \,\mathrm{d}y.
	\end{align*}
	From $|x-\bar{y}| \geq x_2+y_2 \geq y_2$ we have
	\begin{equation*}
		|K_{11}|\leq C \int_0^{2x_2} \frac {\sup_{y_1 \in [x_1/2,2x_1]}|\partial_1 \theta(y_1,y_2)|}{x_2+y_2} \int_0^{2x_2} \frac {1}{|x - y|}\, \mathrm{d}y_1 \,\mathrm{d}y_2
	\end{equation*} and
	\begin{equation*}
		|K_{12}|\leq C \int_{2x_2}^{2x_1} \frac {\sup_{y_1 \in [x_1/2,2x_1]}|\partial_1 \theta(y_1,y_2)|}{x_2+y_2} \int_0^{2x_2} \frac {1}{|x - y|}\, \mathrm{d}y_1 \,\mathrm{d}y_2.
	\end{equation*}
	By Gagliardo-Nirenberg interpolation inequality with $y_2 \leq 2x_1$, we can see that
	\begin{align*}
		y_2^{-\frac 12} {\sup_{y_1 \in [x_1/2,2x_1]}|\partial_1 \theta(y_1,y_2)|} &\leq C (\| \partial_1^2 \theta(\cdot,y_2) \|_{L^2(x_1/2,2x_1)} + y_2^{-1} \| \partial_1 \theta(\cdot,y_2) \|_{L^2(x_1/2,2x_1)}) ,
	\end{align*}
	where the constant $C>0$ does not depend on $x_1$. On the other hand,
	\begin{equation}\label{log_est}
		\begin{aligned}
			\int_0^{2x_1} \frac 1{|x-y|} \,\mathrm{d}y_1 &= 2\int_0^{x_1} \frac 1{\sqrt{\tau^2+(x_2-y_2)^2}} \,\mathrm{d}\tau \\
			&= 2\log \left(x_1+\sqrt{x_1^2+(x_2-y_2)^2} \right) - 2\log |x_2-y_2| \\
			%	&= 2\log\left( \frac{x_1}{|x_2-y_2|} + \sqrt{\frac{x_1^2}{(x_2-y_2)^2} + 1} \right) \\
			&\leq C\log\left( 1+ \frac{x_1}{|x_2-y_2|} \right).
		\end{aligned}
	\end{equation}
	Hence, with $y_2 \leq 2x_1$ and H\"{o}lder's inequality, we infer that
	\begin{align*}
		|K_{11}|&\leq C (\| \nabla^2 \theta \|_{L^2([0,1]^2)} + \| y_2^{-1} \theta(y) \|_{L^2([0,1]^2)}) \left\{ \int_0^{2x_2} \frac 1{x_2+y_2} \left| \log\left( 1+\frac{x_1}{|x_2-y_2|} \right) \right|^2 \mathrm{d}y_2 \right\}^{\frac 12}
	\end{align*}
	and
	\begin{align*}
		|K_{12}|&\leq C (\| \nabla^2 \theta \|_{L^2(R(x))} + \| y_2^{-1} \theta(y) \|_{L^2(R(x))}) \left\{ \int_0^{2x_2} \frac 1{x_2+y_2} \left| \log\left( 1+\frac{x_1}{|x_2-y_2|} \right) \right|^2 \mathrm{d}y_2 \right\}^{\frac 12}.
	\end{align*}
	Using Hardy's inequality and that
	\begin{align*}
		\int_0^{2x_2} \frac 1{x_2+y_2} \left| \log\left( 1+\frac{x_1}{|x_2-y_2|} \right) \right|^2 \mathrm{d}y_2  &\leq \frac 1{x_2} \int_{0}^{2x_2} \left( \log \frac{2x_1}{|x_2-y_2|} \right)^2 \mathrm{d}y_2 \\
		&= \frac 2{x_2}\int_{0}^{x_2} \left( \log \frac {2x_1}{x_2 - y_2} \right)^2 \mathrm{d}y_2 \\
		%&= \frac {2(x_2 - y_2)}{x_2} \left(-2 + 2\log \frac{x_2 - y_2}{2x_1} - \left|\log \frac{x_2 - y_2}{2x_1} \right|^2 \right) \Bigg|_{y_2=0}^{y_2=x_2} \\
		&= 4 + 4 \log \frac{2x_1}{x_2} + 2 \left(\log \frac{2x_1}{x_2} \right)^2 \\
		&\leq C \left(1+\log \frac{x_1}{x_2} \right)^2,
	\end{align*}
	we obtain
	\begin{equation*}
		|K_{11}|\leq C \| \nabla^2 \theta \|_{L^2([0,1]^2)} \left( 1+\log \frac{x_1}{x_2} \right).
	\end{equation*}
	By the following estimate
	\begin{align*}
		\int_{2x_2}^{2x_1} \frac 1{x_2+y_2} \left| \log\left( 1+\frac{x_1}{|x_2-y_2|} \right) \right|^2 \mathrm{d}y_2  &\leq \int_{2x_2}^{2x_1} \frac1{y_2-x_2} \left( \log \frac{3x_1}{y_2-x_2} \right)^2 \mathrm{d}y_2 \\
		%&\leq \frac {-1}3 \left| \log \frac{3x_1}{y_2-x_2} \right|^3 \bigg|_{y_2=2x_2}^{y_2=2x_1} \\
		&= \frac 13 \left( \log \frac{3x_1}{x_2} \right)^3 - \frac 13 \left( \log \frac{3x_1}{2x_1-x_2} \right)^3 \\
		&\leq C \left(1+ \log \frac{x_1}{x_2} \right)^3,
	\end{align*}
	we have
	\begin{equation*}
		|K_{12}|\leq C (\| \nabla^2 \theta \|_{L^2(R(x))} + \| y_2^{-1} \theta(y) \|_{L^2(R(x))}) \left( 1+\log \frac{x_1}{x_2} \right)^{\frac 32}.
	\end{equation*}
	This implies
	\begin{equation*}
		|K_{1}| \leq C \| \nabla^2 \theta \|_{L^2([0,1]^2)} \left( 1+\log \frac{x_1}{x_2} \right) + C (\| \nabla^2 \theta \|_{L^2(R(x))} + \| y_2^{-1} \partial_1 \theta(y) \|_{L^2(R(x))}) \left( 1+\log \frac{x_1}{x_2} \right)^{\frac 32} ,
	\end{equation*}
	and collecting the estimates for $K_{1}$, $K_{2}$, and $K_3$, we obtain \eqref{eq:K-goal}. Therefore, we arrive at \begin{gather*}
		\left|\frac {I_3(0)}{x_2} + 6 \int_{Q(x)} \frac { y_1 y_2 }{|y|^5} \theta(y) \mathrm{d}y \right| \\
		\leq C \| \nabla^2 \theta \|_{L^2([0,1]^2)} \left( 1+\log \frac{x_1}{x_2} \right) + C (\| \nabla^2 \theta \|_{L^2(R(x))} + \| y_2^{-1} \partial_1 \theta(y) \|_{L^2(R(x))}) \left( 1+\log \frac{x_1}{x_2} \right)^{\frac 32} .
	\end{gather*}	
	Using \eqref{AB_est}, we can estimate
	\begin{equation*}
		I_4(0) = - 4x_2 \int_{[0,1]^2} \frac { y_2(x_1 + y_1) (|x + y|^2 + |x + y| |x - \tilde{y}| + |x - \bar{y}|^2)}{|x + y|^3 |x - \tilde{y}|^3 (|x + y| + |x - \tilde{y}|)} \theta(y) \mathrm{d}y  
	\end{equation*}
	similarly with $I_3(0)$. Hence we have \eqref{u2}, and this completes the proof.
\end{proof}

\begin{lemma}\label{key_lem2}
	Moreover, we have
	\begin{equation}\label{u1_2}
		\left|\frac {u_1(x)}{x_1} - 12\int_{Q(x)} \frac { y_1 y_2 }{|y|^5} \theta(y) \mathrm{d}y\right| \le B_4(x)
	\end{equation}
	and
	\begin{equation}\label{u2_2}
		\left|\frac {u_2(x)}{x_2} + 12\int_{Q(x)} \frac { y_1 y_2 }{|y|^5} \theta(y) \mathrm{d}y  \right| \le \left( 1+\log \frac{x_1}{x_2} \right) B_5(x) + \left( 1+\log \frac{x_1}{x_2} \right)^2 B_6(x),
	\end{equation}
	where $B_4$, $B_5$, $B_6$ satisfy
	\begin{equation*}
		|B_4(x)| + |B_5(x)| \leq  C(\| \nabla \theta \|_{L^\infty([0,1]^2)} + \| \theta \|_{L^\infty([0,1]^2)})
	\end{equation*} and
	\begin{equation*}
		|B_6(x)| \leq C \| \nabla \theta \|_{L^\infty(R(x))}.
	\end{equation*}
\end{lemma}

\begin{proof}
	We follow the proof of Lemma~\ref{key_lem}. To obtain \eqref{u1_2}, we recall \eqref{far_est} and have that
	\begin{equation*}
		|u_1(x)| = \Bigg| \sum_{n \in \mathbb{Z}^2} (I_1(n) + I_2(n)) \Bigg| \leq C x_1 \| \theta \|_{L^\infty([0,1]^2)} + I_1(0) + I_2(0).
	\end{equation*} We estimate 
	\begin{equation*}
		I_1(0) = -4x_1 \int_{[0,1]^2} \frac { y_1(x_2 - y_2) (|x - y|^2 + |x - y| |x - \tilde{y}| + |x - \tilde{y}|^2)}{|x - y|^3 |x - \tilde{y}|^3 (|x - y| + |x - \tilde{y}|)} \theta(y) \,\mathrm{d}y
	\end{equation*} for each set $R(2x)$ and $[0,2x_1] \times [0,1]$.
	
	\medskip
	
	\noindent \textit{(i) $Q(x)$} 
	
	\medskip
	
	\noindent Using the notation $J_1$ and $J_2$ in Lemma~\ref{key_lem}, it sufficies to obtain that
	\begin{equation}\label{eq:J-goal_2}
		\left| J_1 - \frac{3}{2} \int_{Q(x)} \frac{y_1y_2}{|y|^5}\theta(y) \,\mathrm{d}y \right| + |J_2| \le C \Vert \nabla \theta \Vert_{L^\infty([0,1]^2)}.
	\end{equation}
	We already showed that \begin{equation*}
		|J_2| \le C|x| \int_{Q(x)} \frac 1{|y|^3} \frac{|\theta(y)|}{|y|}\, \mathrm{d}y .
	\end{equation*}
	By \eqref{eq:equiv-y} and H\"{o}lder's inequaltiy, we have
	\begin{equation*}
		|x| \int_{Q(x)} \frac 1{|y|^3} \frac{|\theta(y)|}{|y|}\, \mathrm{d}y \leq |x| \left( \int_{|x|}^\infty \frac 1{r^2} \, \mathrm{d}r \right) \Vert |y|^{-1}\theta(y) \Vert_{L^\infty(Q(x))} \le \Vert |y|^{-1}\theta(y) \Vert_{L^\infty([0,1]^2)}.
	\end{equation*}
	Due to that $\theta$ vanishes near the axis, it follows
	\begin{equation*}
		|J_2| \leq C \| \nabla \theta \|_{L^\infty([0,1]^2)}.
	\end{equation*}
	Separating $J_1=J_{11}+J_{12}+J_{13}$ as in the proof of Lemma~\ref{key_lem}, we can prove that 
	\begin{equation*}
		\begin{split}
			\left|J_{1k} - \frac{1}{2} \int_{Q(x)} \frac{y_1y_2}{|y|^5}\theta(y) \,\mathrm{d}y \right| \le C|x|\int_{Q(x)}  \frac{1}{|y|^3} \frac{|\theta(y)|}{|y|}\,\mathrm{d}y \leq C \| \nabla \theta \|_{L^\infty([0,1]^2)}
		\end{split}
	\end{equation*} for each $k = 1, 2, 3$. Therefore, \eqref{eq:J-goal_2} is obtained.
	
	\medskip
	
	\noindent \textit{(ii) $[0,2x_1] \times [2x_1,1]$} 
	
	\medskip
	
	\noindent In \eqref{est_ii} we observed that
	\begin{gather*}
		\left|-\int_{[0,2x_1] \times [2x_1,1]} \frac { y_1(x_2 - y_2) (|x - y|^2 + |x - y| |x - \tilde{y}| + |x - \tilde{y}|^2)}{|x - y|^3 |x - \tilde{y}|^3 (|x - y| + |x - \tilde{y}|)} \theta(y) \,\mathrm{d}y \right| \\
		\leq C \int_{[0,2x_1] \times [2x_1,1]} \frac {1}{y_2^2} \frac {\theta(y)}{y_2} \,\mathrm{d}y.
	\end{gather*}
	Since
	\begin{equation}\label{est_ii_2}
		\begin{aligned}
			\int_{[0,2x_1] \times [2x_1,1]} \frac {1}{y_2^2} \frac {\theta(y)}{y_2} \,\mathrm{d}y \leq C \left( \int_{[0,2x_1] \times [2x_1,1]} \frac 1{y_2^2} \,\mathrm{d}y \right) \| y_2^{-1} \theta(y) \|_{L^\infty([0,1]^2)} \leq C \| \nabla \theta \|_{L^\infty([0,1]^2)} ,
		\end{aligned}
	\end{equation}
	we have
	\begin{equation*}
		\left|-\int_{[0,2x_1] \times [2x_1,1]} \frac { y_1(x_2 - y_2) (|x - y|^2 + |x - y| |x - \tilde{y}| + |x - \tilde{y}|^2)}{|x - y|^3 |x - \tilde{y}|^3 (|x - y| + |x - \tilde{y}|)} \theta(y) \,\mathrm{d}y \right| \leq C \| \nabla \theta \|_{L^\infty([0,1]^2)}.
	\end{equation*}
	
	\medskip
	
	\noindent \textit{(iii) $[0,2x_1]^2$} 
	
	\medskip
	
	\noindent We recall that
	\begin{align*}
		&-\int_{[0,2x_1]^2} \frac { y_1(x_2 - y_2) (|x - y|^2 + |x - y| |x - \tilde{y}| + |x - \tilde{y}|^2)}{|x - y|^3 |x - \tilde{y}|^3 (|x - y| + |x - \tilde{y}|)} \theta(y) \,\mathrm{d}y \\
		&\hphantom{\qquad\qquad}= \frac 1{x_1} \int_{[0,2x_1]^2} \left( \frac 1{|x-y|} - \frac 1{|x-\tilde{y}|} \right) \partial_2 \theta(y)\, \mathrm{d}y \\
		&\hphantom{\qquad\qquad\qquad\qquad}- \frac 1{x_1} \int_0^{2x_1} \left( \frac 1{|(x_1-y_1,x_2-2x_1)|} - \frac 1{|(x_1+y_1,x_2-2x_1)|} \right) \theta(y_1,2x_1) \,\mathrm{d}y_1.
	\end{align*}
	Using H\"{o}lder's inequality, we have
	\begin{align*}
		\left| \frac 1{x_1} \int_{[0,2x_1]^2} \left( \frac 1{|x-y|} - \frac 1{|x-\tilde{y}|} \right) \partial_2 \theta(y)\, \mathrm{d}y \right| &\leq \frac 2{x_1} \left( \int_{[0,2x_1]^2} \frac {1}{|x-y|}\, \mathrm{d}y \right) \| \partial_2 \theta \|_{L^\infty([0,1]^2)} \\
		&\leq C \| \nabla \theta \|_{L^\infty([0,2x_1]^2)}.
	\end{align*}
	From H\"{o}lder's inequality and the mean value theroem, it follows 
	\begin{equation*}
		\left| - \frac 1{x_1} \int_0^{2x_1} \left( \frac 1{|(x_1-y_1,x_2-2x_1)|} - \frac 1{|(x_1+y_1,x_2-2x_1)|} \right) \theta(y_1,2x_1) \,\mathrm{d}y_1 \right| \leq C \| \nabla \theta \|_{L^\infty([0,1]^2)}.
	\end{equation*}
	Therefore, we obtain that
	\begin{equation*}
		\left| -\int_{[0,2x_1]^2} \frac { y_1(x_2 - y_2) (|x - y|^2 + |x - y| |x - \tilde{y}| + |x - \tilde{y}|^2)}{|x - y|^3 |x - \tilde{y}|^3 (|x - y| + |x - \tilde{y}|)} \theta(y) \,\mathrm{d}y \right| \leq C \| \nabla^2 \theta \|_{L^2([0,1]^2)}.
	\end{equation*}
	\medskip
	Combining the above estimates, it follows that 
	\begin{align*}
		\left|\frac {I_1(0)}{x_1} - 6 \int_{Q(x)} \frac { y_1 y_2 }{|y|^5} \theta(y) \,\mathrm{d}y \right| \le C \| \nabla \theta \|_{L^\infty([0,1]^2)}.
	\end{align*}
	In a similar way, we can show that
	\begin{equation*}
		I_2(0) = 4x_1 \int_{[0,1]^2} \frac { y_1(x_2 + y_2) (|x + y|^2 + |x + y| |x - \bar{y}| + |x - \bar{y}|^2)}{|x + y|^3 |x - \bar{y}|^3 (|x - y| + |x - \bar{y}|)} \theta(y) \mathrm{d}y 
	\end{equation*}
	satisfies \begin{align*}
		\left|\frac {I_2(0)}{x_2} - 6 \int_{R(2x)} \frac { y_1 y_2 }{|y|^5} \theta(y) \mathrm{d}y \right| \le C \| \nabla \theta \|_{L^\infty([0,1]^2)}.
	\end{align*} We omit the details. Thus  we have \eqref{u1_2}.
	
	\medskip
	
	Now we estimate $u_2$. We already know that
	\begin{equation*}
		|u_2(x)| = \Bigg| \sum_{n \in \mathbb{Z}^2} (I_3(n) + I_4(n)) \Bigg| \leq C x_1 \| \theta \|_{L^\infty([0,1]^2)} + I_3(0) + I_4(0).
	\end{equation*} To estimate
	\begin{equation*}
		I_3(0) = 4x_2 \int_{[0,1]^2} \frac { y_2(x_1 - y_1) (|x - y|^2 + |x - y| |x - \bar{y}| + |x - \bar{y}|^2)}{|x - y|^3 |x - \bar{y}|^3 (|x - y| + |x - \bar{y}|)} \theta(y) \mathrm{d}y,
	\end{equation*}
	we divide the domain into four regions as $[0,1]^2 = Q(x) \cup [0,x_1/2] \times [0,2x_1] \cup [x_1/2,2x_1] \times [0,2x_1] \cup [0,2x_1] \times [2x_1,1]$ and estimate the integral in each region. 
	
	\medskip
	
	\noindent \textit{(i) $Q(x)$} 
	
	\medskip
	
	\noindent In the case of $y \in Q(x)$, recalling that \eqref{eq:equiv-y} holds, we can prove 
	\begin{equation}\label{eq:J-goal2_2}
		\begin{gathered}
			\left| \int_{Q(x)} \frac { y_2(x_1 - y_1) (|x - y|^2 + |x - y| |x - \bar{y}| + |x - \bar{y}|^2)}{|x - y|^3 |x - \bar{y}|^3 (|x - y| + |x - \bar{y}|)} \theta(y) \,\mathrm{d}y + \frac{3}{2} \int_{Q(x)} \frac{y_1y_2}{|y|^5}\theta(y) \,\mathrm{d}y \right|\\
			\le C \Vert \nabla \theta \Vert_{L^\infty([0,1]^2)}.
		\end{gathered}
	\end{equation}  We shall not give a proof as it requires little adjustment from \eqref{eq:J-goal_2}.
	
	\medskip
	
	\noindent \textit{(ii) $[0,2x_1] \times [2x_1,1]$} 
	
	\medskip
	
	\noindent Since $y \in [0,2x_1] \times [2x_1,1]$ satisfies
	\begin{equation*}
		\frac 12{y_2} \leq |x-y| \leq 2y_2, \qquad {y_2} \leq |x-\bar{y}| \leq 2y_2,
	\end{equation*} with \eqref{est_ii_2} we can show that	
	\begin{equation*}
		\left|\int_{[0,2x_1] \times [2x_1,1]} \frac { y_2(x_1 - y_1) (|x - y|^2 + |x - y| |x - \bar{y}| + |x - \bar{y}|^2)}{|x - y|^3 |x - \bar{y}|^3 (|x - y| + |x - \bar{y}|)} \theta(y) \,\mathrm{d}y \right| \leq C \| \nabla \theta \|_{L^\infty([0,1]^2)} .
	\end{equation*}
	%	\begin{align*}
	%	&\left|\int_{[0,2x_1] \times [2x_1,1]} \frac { y_2(x_1 - y_1) (|x - y|^2 + |x - y| |x - \bar{y}| + |x - \bar{y}|^2)}{|x - y|^3 |x - \bar{y}|^3 (|x - y| + |x - \bar{y}|)} \theta(y) \,\mathrm{d}y \right| \\
	%	&\hphantom{\qquad\qquad}\leq C \left|\int_{[0,2x_1] \times [2x_1,1]} \frac {1}{|x - \bar{y}|^2} \frac {\theta(y)}{y_2} \,\mathrm{d}y \right| \\
	%	&\hphantom{\qquad\qquad}\leq C \left( \int_{[0,2x_1] \times [2x_1,1]} \frac 1{|x_2+y_2|^2} \,\mathrm{d}y \right)^{\frac 12} \| y_2^{-1} \theta(y) \|_{L^2([0,1]^2)} \\
	%	&\hphantom{\qquad\qquad}\leq C \| \nabla \theta \|_{L^\infty([0,1]^2)} .
	%	\end{align*}
	
	\medskip
	
	\noindent \textit{(iii) $[0,x_1/2] \times [0,2x_1]$} 
	
	\medskip
	
	\noindent From that $y \in [0,x_1/2] \times [0,2x_1]$ implies
	\begin{equation*}
		\frac 12 x_1 \leq |x-y| \leq 4x_1, \qquad \frac 12 x_1 \leq |x-\bar{y}| \leq 4x_1,
	\end{equation*} we can see that
	\begin{equation*}
		\left|\int_{[0,x_1/2] \times [0,2x_1]} \frac { y_2(x_1 - y_1) (|x - y|^2 + |x - y| |x - \bar{y}| + |x - \bar{y}|^2)}{|x - y|^3 |x - \bar{y}|^3 (|x - y| + |x - \bar{y}|)} \theta(y) \,\mathrm{d}y \right| \\
		\leq C \| \nabla \theta \|_{L^\infty([0,1]^2)} .
	\end{equation*}
	
	\medskip
	
	\noindent \textit{(iv) $[x_1/2,2x_1] \times [0,2x_1]$} 
	
	\medskip
	
	\noindent Recalling the notation $K_1$, $K_2$, and $K_3$ in the proof of Lemma~\ref{key_lem}, we claim
	\begin{equation}\label{K1}
		|K_1| \leq C \| \nabla \theta \|_{L^\infty([0,1]^2)} \left( 1+\log \frac{x_1}{x_2} \right) + C \| \nabla \theta \|_{L^\infty(R(x))} \left( 1+\log \frac{x_1}{x_2} \right)^2
	\end{equation} and
	\begin{equation}\label{K23}
		|K_2| + |K_3| \leq C \| \nabla \theta \|_{L^\infty([0,1]^2)}.
	\end{equation} 
	As in \eqref{K2}, we have with the mean value theorem that
	\begin{align*}
		|K_2| \leq C x_1^{-1} \| \theta \|_{L^\infty([0,2x_1]^2)} \leq C \| \nabla \theta \|_{L^\infty([0,1]^2)},
	\end{align*} and similarly,
	\begin{equation*}
		|K_3|\leq C \| \nabla \theta \|_{L^\infty([0,1]^2)}.
	\end{equation*}
	Hence, \eqref{K23} follows. We reacll
	\begin{equation*}
		K_1 = K_{11} + K_{12}
	\end{equation*} where
	\begin{align*}
		K_{11} &= \int_{[x_1/2,2x_1] \times [0,2x_2]} \frac {-4 y_2 \partial_1 \theta(y)}{|x - y| |x - \bar{y}| (|x - y| + |x - \bar{y}|)} \,\mathrm{d}y, \\
		K_{12} &= \int_{[x_1/2,2x_1] \times [2x_2,2x_1]} \frac {-4 y_2 \partial_1 \theta(y)}{|x - y| |x - \bar{y}| (|x - y| + |x - \bar{y}|)} \,\mathrm{d}y.
	\end{align*}
	From H\"{o}lder's inequality and \eqref{log_est}, we can deduce that
	\begin{equation*}
		|K_{11}|\leq C  \| \nabla \theta \|_{L^\infty([0,1]^2)} \int_0^{2x_2} \frac {1}{x_2+y_2} \log\left( 1+ \frac{x_1}{|x_2-y_2|} \right) \mathrm{d}y_2
	\end{equation*} and
	\begin{equation*}
		|K_{12}|\leq C \| \nabla \theta \|_{L^\infty(R(x))} \int_{2x_2}^{2x_1} \frac {1}{x_2+y_2} \log\left( 1+ \frac{x_1}{|x_2-y_2|} \right)\mathrm{d}y_2 .
	\end{equation*}
	Since
	\begin{align*}
		\int_0^{2x_2} \frac 1{x_2+y_2} \log \left( 1+\frac{x_1}{|x_2-y_2|} \right) \mathrm{d}y_2 &\leq \frac 1{x_2} \int_{0}^{2x_2} \log \frac{2x_1}{|x_2-y_2|}\, \mathrm{d}y_2 \\
		&= 2 \log \frac {2x_1}{x_2} -2 \\
		&\leq C \left(1+\log \frac{x_1}{x_2} \right)
	\end{align*}
	and
	\begin{align*}
		\int_{2x_2}^{2x_1} \frac {1}{x_2+y_2} \log\left( 1+ \frac{x_1}{|x_2-y_2|} \right) \mathrm{d}y_2 &\leq \int_{2x_2}^{2x_1} \frac1{y_2-x_2} \log \frac {3x_1}{y_2-x_2} \, \mathrm{d}y_2 \\
		%&\leq \frac {-1}3 \left| \log \frac{3x_1}{y_2-x_2} \right|^3 \bigg|_{y_2=2x_2}^{y_2=2x_1} \\
		&= \frac 12 \left| \log \frac{3x_1}{x_2} \right|^2 - \frac 12 \left| \log \frac{3x_1}{2x_1-x_2} \right|^2 \\
		&\leq C \left(1+ \log \frac{x_1}{x_2} \right)^2,
	\end{align*}
	it follows
	\begin{equation*}
		|K_{11}|\leq C \| \nabla \theta \|_{L^\infty([0,1]^2)} \left( 1+\log \frac{x_1}{x_2} \right)
	\end{equation*} and
	\begin{equation*}
		|K_{12}|\leq C \| \nabla \theta \|_{L^\infty(R(x))} \left( 1+\log \frac{x_1}{x_2} \right)^2.
	\end{equation*}
	This shows \eqref{K1}.
	Combining the estimates, we arrive at \begin{gather*}
		\left|\frac {I_3(0)}{x_2} + 6 \int_{Q(x)} \frac { y_1 y_2 }{|y|^5} \theta(y) \mathrm{d}y \right| 
		\leq C \| \nabla \theta \|_{L^\infty([0,1]^2)} \left( 1+\log \frac{x_1}{x_2} \right) + C \| \nabla \theta \|_{L^\infty(R(x))} \left( 1+\log \frac{x_1}{x_2} \right)^2.
	\end{gather*}
	Using \eqref{AB_est}, we can estimate
	\begin{equation*}
		I_4(0) = - 4x_2 \int_{[0,1]^2} \frac { y_2(x_1 + y_1) (|x + y|^2 + |x + y| |x - \tilde{y}| + |x - \bar{y}|^2)}{|x + y|^3 |x - \tilde{y}|^3 (|x + y| + |x - \tilde{y}|)} \theta(y) \mathrm{d}y  
	\end{equation*}
	similarly with $I_3(0)$. Hence we have \eqref{u2_2}, and this completes the proof.
\end{proof}

\section{Nonexistence}\label{sec:nonexist}

In this section, we prove Theorem \ref{thm:nonexist}.
We begin with a simple uniqueness result which in particular guarantees that the hypothetical solution in $L^\infty_t H^2$ satisfies the same symmetries with the initial data. 
\begin{proposition}\label{prop:unique}
	Given $\tht_0\in H^2$ and $T>0,$ there exists at most one solution belonging to $L^\infty([0,T];H^2)$ to \eqref{eq:SQG} with initial data $\tht_0$. 
\end{proposition}

\begin{proof}
	The proof can be given by simply adapting Yudovich's inequalities derived in (\cite{Y1,Y2}). This statement could be found in \cite{Bed} as well. 
\end{proof}

\begin{proof}[Proof of Theorem \ref{thm:nonexist}] For convenience, we shall divide the proof into several parts.

\medskip

\noindent \textbf{1. Velocity and flow map: an $L^\infty_{t}H^{2}$--solution is {Lagrangian}}. Assume that we are given a solution to \eqref{eq:SQG} satisfying $$\sup_{t\in[0,T]}\nrm{\tht(t,\cdot)}_{H^{2}}\le M.$$ Then, by the Sobolev embedding, $u = \nb^\perp(-\lap)^{-\frac12}\tht$  satisfies \begin{equation*}
	\begin{split}
		\sup_{t\in[0,T]}(\nrm{\nb u(t,\cdot)}_{BMO} + \nrm{u(t,\cdot)}_{W^{1,1}})\le C \sup_{t\in[0,T]}\nrm{u(t,\cdot)}_{H^{2}}\le CM
	\end{split}
\end{equation*} with some absolute constant $C>0$. In particular, $u$ is log-Lipschitz: for any $x,y \in \bbT^2$, we have \begin{equation}\label{eq:log-Lipschitz}
	\begin{split}
		\sup_{t\in[0,T]} |u(t,x)-u(t,y)| \le CM|x-y| \ln\left( 10 + \frac{1}{|x-y|} \right).
	\end{split}
\end{equation} On the time interval $[0,T]$, we consider the flow map $\Phi(t,\cdot):\bbT^2\rightarrow\bbT^2$ defined by \begin{equation}  \label{eq:flow-def}
	\left\{
	\begin{aligned}
		&\frac{d}{dt} \Phi (t,x) = u(t,\Phi(t,x)), \\
		&\Phi(0,x) = x. 
	\end{aligned}
	\right.
\end{equation} It is well-known that under the estimate \eqref{eq:log-Lipschitz}, there is a unique solution to the ODE \eqref{eq:flow-def} for any $x\in\bbT^2$ (\cite{MB,MP}). The solution $\Phi$ satisfies the estimate \begin{equation}\label{eq:flow-estimate}
	\begin{split}
		|x-y|^{\exp(CMt)} \le |\Phi(t,x)-\Phi(t,y)| \le |x-y|^{\exp(-CMt)}
	\end{split}
\end{equation} for some absolute constant $C>0$, uniformly in $x, y \in \bbT^2$  satisfying $|x-y|<\frac12$ and $t\in[0,T]$. We have the representation \begin{equation*}
	\begin{split}
		\tht(t,\Phi(t,x)) = \tht_0(x). 
	\end{split}
\end{equation*} The estimate \eqref{eq:flow-estimate} shows that for each $t\in[0,T]$, $\Phi(t,\cdot)$ is a H\"older continuous homeomorphism $\bbT^2\rightarrow\bbT^2$, and we denote the inverse map by $\Phi_t^{-1}$. Then, with this notation, we have \begin{equation*}
	\begin{split}
		\tht(t,x) = \tht_0(\Phi_t^{-1}(x)). 
	\end{split}
\end{equation*} The inverse map  $\Phi_t^{-1}$ is again  H\"older continuous. As an immediate consequence, we have that if $\tht_{0}$ is an odd function with respect to both axes and satisfies \begin{equation*}
\begin{split}
	\supp(\tht_{0}) \cup  \{ x: x_1=0 \mbox{ or } x_2 = 0 \} \subset \{ (0,0) \},
\end{split}
\end{equation*} then the same properties are satisfied by $\tht(t,\cdot)$, as long as $\tht \in L^{\infty}([0,t];H^{2} )$. Indeed, the uniqueness assertion from Proposition \ref{prop:unique} guarantees that $\tht(t,\cdot)$ is odd with respect to both axes. Furthermore, H\"older continuity of the flow map and its inverse ensures that $\tht(t,\cdot)$ vanishes near the axes, possibly except at the origin. Therefore, the last assumption in Lemma \ref{key_lem} is satisfied. 

\medskip

\noindent \textbf{2. Choice of initial data}.  We fix some smooth bump function $\phi: \bbR^2\rightarrow \bbR_{\ge0}$ satisfying the following properties: 
\begin{itemize}
	\item $\phi$ is $C^\infty$-smooth and radial. 
	\item $\phi$ is supported in $B_{0}(\frac{1}{8})$ and $\phi = 1$ in $B_{0}(\frac{1}{32})$. 
\end{itemize} Then, we define 
\begin{equation}\label{eq:nonexist-data}
	\begin{split}
		\tht_0 :=  \sum_{n=n_{0}}^{\infty}  n^{-\alp} \tht^{(n)}_{0, loc}
	\end{split}
\end{equation}
for some $1/2 < \alpha < 3/4$, where
\begin{equation*}
	\begin{split}
		\tht^{(n)}_{0, loc}(x) := 4^{-n} \phi( 4^n(x_1-4^{-n-1}, x_2-2^{-1}4^{-n-1} ) )
	\end{split}
\end{equation*}
for $x \in [0,1]^2$. The precise value of $\alp$ will be determined later, but for now let us just mention that it will be taken slightly larger than $\frac12$. Next, let us extend each of $\tht^{(n)}_{0, loc}$ (and similarly $\tht_{0}$) to $\bbT^2$ as an odd function with respect to both axes. Note that by taking $n_{0}\ge 1$ sufficiently large in a way depending only on $\eps>0$, we can guarantee that 
$$\| \theta_0 \|_{H^{2}\cap W^{1,\infty}(\bbT^2)} <\eps.$$
Towards a contradiction, we shall assume that there exists $M>0$ and $T>0$ such that \begin{equation}\label{eq:sol-bound}
	\begin{split}
		\sup_{t \in [0,T]}\| \theta(t) \|_{ \dot{H}^{2}  (\bbT^2)} \leq M.
	\end{split}
\end{equation} For simplicity, we shall assume that $M\ge 1$. Observe that the assumptions in  the Key Lemma \ref{key_lem} are satisfied by $\tht_{0}$. Recalling the discussion above, we have that $\tht(t,\cdot)$ is odd with respect to both axes and vanishes near the axes, except at the origin.

\medskip

\noindent \textbf{3. Preliminary bounds on the solution}. Let us remark in advance that in the following proof, we shall take $T>0$ to be smaller, if necessary, to satisfy $T \le c/M$ for some small absolute constant $c>0$.  We shall begin with a simple result:
\begin{lemma}\label{lem:basic}
	Assume that $\tht$ is a solution satisfying \eqref{eq:sol-bound} with initial data \eqref{eq:nonexist-data}. Then,  by redefining $T$ to satisfy $T \le c/M$ if necessary, we have \begin{equation*}
		\begin{split}
			\tht(t,y) = 0, \qquad 0 \le y_1 \le y_2, \qquad 0 \le t \le T. 
		\end{split}
	\end{equation*}
\end{lemma} \begin{proof}
Since $\tht(t,\Phi(t,x)) = \tht_0(x)$, to prove the claim, it suffices to show that for $x \in \supp( \tht_0 )\backslash \{ (0,0) \}$, $\Phi_{2}(t,x) \le \Phi_{1}(t,x)$ for $0\le t \le T$. Let us fix some $x \in \supp( \tht_0 )\backslash \{ (0,0) \}$. Then, from the choice of initial data, we have $2x_2 \le x_1$. From continuity in time of the flow map, there exists some $0<T^* \le T$ such that $\Phi_{2}(t,x) < \Phi_{1}(t,x)$ for $0\le t < T^*$. Then, on this time interval, Key Lemma is applicable for $\Phi(t,x)$ and we compute \begin{equation*}
	\begin{split}
		\frac{d}{dt}\left( \frac{ \Phi_2(t,x)}{\Phi_1(t,x)} \right) & = \frac{ \Phi_2(t,x)}{\Phi_1(t,x)}\left(  \frac{u_2(t,\Phi(t,x))}{\Phi_2(t,x)} -   \frac{u_1(t,\Phi(t,x))}{\Phi_1(t,x)} \right) \\
		& \le C\frac{ \Phi_2(t,x)}{\Phi_1(t,x)}\left( |B_1(\Phi(t,x))| + |B_2(\Phi(t,x))| +|B_3(\Phi(t,x))|   \right) \\
		& \le CM\frac{ \Phi_2(t,x)}{\Phi_1(t,x)}. 
	\end{split}
\end{equation*} Therefore, we actually obtain \begin{equation*}
\begin{split}
	\frac{ \Phi_2(t,x)}{\Phi_1(t,x)} \le \frac{1}{2} \exp\left( CMt \right) < \frac34 
\end{split}
\end{equation*} on $t \in [0,T^*]$, as long as $T^* \le c/M$ for $c>0$ depending only on $C$. This bootstrap procedure allows us to get $\Phi_2/\Phi_1 < 3/4$ uniformly in $x \in \supp(\tht_0)\backslash\{(0,0)\}$ by the time $\min\{ T,c/M\} = T$. 
\end{proof}

The above lemma guarantees that on $[0,T]$, Key Lemma is applicable to points in $\supp(\tht(t,\cdot))$. Next, let us prove that by reducing $c>0$ if necessary, the bubbles are ``well-ordered'' with respect to the $x_{1}$-axis for $t \in [0,T]$ with $T \le c/M$. 
\medskip

\noindent \textbf{Claim I.} We have that \begin{equation}\label{eq:claim}
	\begin{split} 
		\sup_{x \in \Omega_n} \Phi_1(t,x) \leq 2\inf_{x \in \Omega_n} \Phi_1(t,x) \qquad \mbox{and} \qquad 2 \sup_{x \in \Omega_{n+1}} \Phi_1(t,x) \leq \inf_{x \in \Omega_n} \Phi_1(t,x)		
	\end{split}
\end{equation} uniformly for all $n \ge n_{0}$ and $t \in [0,T]$, by reducing $T$ to satisfy $T \le c/(1 + M)$ for some small absolute constant $c>0$. 

\medskip

\noindent For simplicity we let $\Omega_n := \mathrm{supp}\,(\tht^{(n)}_{0, loc})$ and
$$\widehat{\Phi}_j^n(t) := \sup_{x \in \Omega_n} \Phi_j(t,x)$$ for $j = 1, 2$. We can prove the \textbf{Claim I} inductively in $n$, using the Key Lemma which gives 
$$\left|\frac {u_1(t,x)}{x_1} - 12\int_{Q(x)} \frac { y_1 y_2 }{|y|^5} \theta(t,y) \mathrm{d}y\right| \le CM.$$ In the proof, we shall take $T>0$ smaller several times, but in a way which is independent of $n$. To begin with, 
for $x \in \Omega_{n_{0}}$ we have
$$\left| \frac {\mathrm{d}}{\mathrm{d}t} \log \Phi_1(t,x)\right| \le CM,$$
thus,
$$\log \Phi_1(t,x) - \log x_1 \geq -CMt.$$
We also have
$$\frac {\mathrm{d}}{\mathrm{d}t} \log \widehat{\Phi}^{n_{0}}_1(t) - \frac {\mathrm{d}}{\mathrm{d}t} \log \Phi_1(t,x) \le 2CM,$$
and thus,
$$\log \widehat{\Phi}^{n_{0}}_1(t) - \log \Phi_1(t,y) \leq 2CMt + (\log \widehat{x}^{n_{0}}_1 - \log x_1).$$
Since $\widehat{x}^{n_{0}}_1 / x_1 < 3/2$, we can take $T>0$ sufficiently small such that
$$2CMT + (\log \widehat{x}^{n_{0}}_1 - \log x_1) \leq \log 2,$$
which implies that
$$\widehat{\Phi}_1^{n_{0}}(t) \leq 2\inf_{x \in \Omega_{n_0}} \Phi_1(t,x)$$
for all $t \in [0,T]$. Indeed, it suffices to take $T= c/(1+M)$ with a small absolute constant $c>0$. 
Note that for $x \in \Omega_{n_{0}}$,
$$\frac {\mathrm{d}}{\mathrm{d}t} \log \Phi_1(t,x) - \frac {\mathrm{d}}{\mathrm{d}t} \log \widehat{\Phi}^{n_{0}+1}_1(t) \geq - 12\int_{\Omega_{n_{0}}} \frac { \Phi_1(t,y) \Phi_2(t,y) }{|\Phi(t,y)|^5} \theta_0(y) \mathrm{d}y - 2CM.$$
With the above estimates, we have
\begin{align*}
	\int_{\Omega_{n_0}} \frac { \Phi_1(t,y) \Phi_2(t,y) }{|\Phi(t,y)|^5} \theta_0(y) \mathrm{d}y &\leq \Big(\sup_{x \in \Omega_{n_{0}}} \frac {x_1}{\Phi_1(t,x)}\Big)^3 \int_{\Omega_{n_{0}}} \frac {\theta_0(y)}{y_1^3} \mathrm{d}y \leq C_0 e^{3CMT}.
\end{align*}
Using it, we obtain
$$\frac {\mathrm{d}}{\mathrm{d}t} \log \Phi_1(t,x) - \frac {\mathrm{d}}{\mathrm{d}t} \log \widehat{\Phi}^{n_0+1}_1(t) \geq - 12C_0 e^{3CMT} - 2CM$$
and
$$\log \Phi_1(t,x) - \log \widehat{\Phi}^{n_{0}+1}_1(t) \geq - 12C_0 e^{3CMT}t - 2CMt + (\log x_1 - \log \widehat{x}^{n_{0}+1}_1).$$
Since $x_1/\widehat{x}^{n_{0}+1}_1 > 2$, we can take $T>0$ sufficiently small such that
$$- 12C_0 e^{3CMT}T - 2CMT + (\log x_1 - \log \widehat{x}^{n_{0}+1}_1) \geq \log 2,$$
which implies that
\begin{equation}\label{bub_12}
	2 \widehat{\Phi}^{n_{0}+1}_1(t) \leq \inf_{x \in \Omega_1} \Phi_1(t,x)
\end{equation}
for all $t \in [0,T]$. Now let $x \in \Omega_{n_{0}+1}$. Then we have
$$ \frac {\mathrm{d}}{\mathrm{d}t} \log \Phi_1(t,x) \geq -CM,$$
thus,
$$\log \Phi_1(t,x) - \log x_1 \geq -CMt.$$ Since $\widehat{x}^{n_{0}+1}_1/x_1$ takes the same value as in the previous case, we see that
$$2CMT + (\log \widehat{x}^{n_{0}+1}_1 - \log x_1) \leq \log 2,$$
and therefore, we have 
$$\widehat{\Phi}_1^{n_{0}+1}(t) \leq 2\inf_{x \in \Omega_{n_{0}+1}} \Phi_1(t,x)$$
for all $t \in [0,T]$.
Note that by \eqref{bub_12},
$$\frac {\mathrm{d}}{\mathrm{d}t} \log \Phi_1(t,x) - \frac {\mathrm{d}}{\mathrm{d}t} \log \widehat{\Phi}^{n_{0}+2}_1(t) \geq - 12\int_{\Omega_{n_{0}+1}} \frac { \Phi_1(t,y) \Phi_2(t,y) }{|\Phi(t,y)|^5} \theta_0(y) \mathrm{d}y - 2CM.$$
With the above estimates, we have
\begin{align*}
	\int_{\Omega_{n_{0}+1}} \frac { \Phi_1(t,y) \Phi_2(t,y) }{|\Phi(t,y)|^5} \theta_0(y) \mathrm{d}y &\leq \Big(\sup_{x \in \Omega_{n_{0}+1}} \frac {x_1}{\Phi_1(t,x)}\Big)^{n_{0}+2} \int_{\Omega_{n_{0}+1}} \frac {\theta_0(y)}{y_1^{n_{0}+2}} \mathrm{d}y \leq C_0 e^{3CMT}.
\end{align*}
Using it, we obtain
$$\frac {\mathrm{d}}{\mathrm{d}t} \log \Phi_1(t,x) - \frac {\mathrm{d}}{\mathrm{d}t} \log \widehat{\Phi}^{n_{0}+2}_1(t) \geq - 12C_0 e^{3CMT} - 2CM$$
and
$$\log \Phi_1(t,x) - \log \widehat{\Phi}^{n_{0}+2}_1(t) \geq - 12C_0 e^{3CMT}t - 2CMt + (\log x_1 - \log \widehat{x}^{n_{0}+2}_1).$$
Since $x_1/\widehat{x}^{n_{0}+2}_1 > 2$ is the same value as in the previous case, it follows
$$- 12C_0 e^{3CMT}T - 2CMT + (\log x_1 - \log \widehat{x}^{n_{0}+2}_1) \geq \log 2$$
and
$$2 \widehat{\Phi}^{n_{0}+2}_1(t) \leq \inf_{x \in \Omega_{n_{0}+1}} \Phi_1(t,x)$$
for all $t \in [0,T]$. Repeating this argument, one can finish the proof of \textbf{Claim I}. \qedsymbol

\medskip

\noindent \textbf{Claim II}. There exists $T>0$ and $C>0$ such that
$$\log \widehat{\Phi}_2^n(T) \leq \log \widehat{x}_2^n - 10 \sum_{n_{0} \leq j \leq n-1} \int_0^T \int_{\Omega_j} \frac { \Phi_1(t,y) \Phi_2(t,y) }{|\Phi(t,y)|^5} \theta_0(y) \mathrm{d}y \mathrm{d}t + CM$$
uniformly for all $n>n_{0}$. 

\medskip

\noindent Recall that 
$$\frac {u_2(x)}{x_2}  \le- 12\int_{Q(x)} \frac { y_1 y_2 }{|y|^5} \theta(t,y) \mathrm{d}y + CM \bigg( \log \frac {x_1}{x_2} \bigg),$$
if $\tht(y) = 0$ for $y$ satisfying $x_1/2 \leq y_1 \leq 2x_1$ and $2x_2 \leq y_2 \leq 1$. According to the order of the bubbles, we have for $x \in \Omega_n$ that
$$\int_{R(2\Phi(t,x))} \frac { y_1 y_2 }{|y|^5} \theta(t,y) \mathrm{d}y = \sum_{n_{0} \leq j \leq n-1} \int_{\Omega_j} \frac { \Phi_1(t,y) \Phi_2(t,y) }{|\Phi(t,y)|^5} \theta_0(y) \mathrm{d}y.$$
And note that
$$\sup_{2\Phi_2(t,x) \geq \widehat{\Phi}_2^n(t)} CM \bigg( \log \frac {\Phi_1(t,x)}{\Phi_2(t,x)} \bigg) \leq CM \bigg( \log \frac {2\widehat{\Phi}_1^n(t)}{\widehat{\Phi}_2^n(t)} \bigg).$$
Thus, we can see that
\begin{align*}
	\frac {\mathrm{d}}{\mathrm{d}t} \log \widehat{\Phi}_2^n(t) &\le- 12 \sum_{n_{0} \leq j \leq n-1} \int_{\Omega_j} \frac { \Phi_1(t,y) \Phi_2(t,y) }{|\Phi(t,y)|^5} \theta_0(y) \mathrm{d}y + CM \bigg( \log \frac {2\widehat{\Phi}_1^n(t)}{\widehat{\Phi}_2^n(t)} \bigg) \\
	&\leq - 12 \sum_{n_{0} \leq j \leq n-1} \int_{\Omega_j} \frac { \Phi_1(t,y) \Phi_2(t,y) }{|\Phi(t,y)|^5} \theta_0(y) \mathrm{d}y + CM \log (2\widehat{\Phi}_1^n(t)) - CM \log \widehat{\Phi}_2^n(t)
\end{align*}
and
\begin{align*}
	\log \widehat{\Phi}_2^n(t) &\leq \log \widehat{x}_2^n - 12 \sum_{n_{0} \leq j \leq n-1} \int_0^t \int_{\Omega_j} \frac { \Phi_1(\tau,y) \Phi_2(\tau,y) }{|\Phi(\tau,y)|^5} \theta_0(y) \mathrm{d}y \mathrm{d}\tau \\
	&\hphantom{\qquad\qquad} + CM \int_0^t \log (2\widehat{\Phi}_1^n(\tau)) \mathrm{d}\tau - CM \int_0^t \log \widehat{\Phi}_2^n(\tau) \mathrm{d}\tau.
\end{align*}
Note that
\begin{align*}
	\log (2\widehat{\Phi}_1^n(t)) &= \log 2 + \log \widehat{\Phi}_1^n(t) \\
	&\leq \log 2 + \log \widehat{x}_1^n + 12 \sum_{n_{0} \leq j \leq n-1} \int_0^t \int_{\Omega_j} \frac { \Phi_1(\tau,y) \Phi_2(\tau,y) }{|\Phi(\tau,y)|^5} \theta_0(y) \mathrm{d}y \mathrm{d}\tau + CMt,
\end{align*}
hence,
$$\int_0^t \log (2\widehat{\Phi}_1^n(\tau)) \mathrm{d}\tau \leq t \log 2 + t \log \widehat{x}_1^n + 12 t \sum_{n_{0} \leq j \leq n-1} \int_0^t \int_{\Omega_j} \frac { \Phi_1(\tau,y) \Phi_2(\tau,y) }{|\Phi(\tau,y)|^5} \theta_0(y) \mathrm{d}y \mathrm{d}\tau + \frac C2 Mt^2.$$
Therefore, with a sufficiently small $T>0$ (independent of $n$), we can complete the proof of \textbf{Claim II}. \qedsymbol

\medskip

\noindent \textbf{4. Almost invariant timescales}. We shall write $\Phi(t,\Omg_{n}) \sim \Omg_{n}$ if \begin{equation*}
	\begin{split}
		\supp( \Phi(t,\Omg_{n}) ) \subset B( (4^{-n-1},2^{-1}4^{-n-1}), 4^{-n-1} ). 
	\end{split}
\end{equation*} Here, $B( (4^{-n-1},2^{-1}4^{-n-1}), 4^{-n-1} )$ denotes the ball of radius $4^{-n-1}$ centered at $(4^{-n-1},2^{-1}4^{-n-1})$. Recall from the definition of initial data that \begin{equation*}
\begin{split}
	\Omg_{n} = B( (4^{-n-1},2^{-1}4^{-n-1}), 2^{-1}4^{-n-1} ). 
\end{split}
\end{equation*} An immediate consequence of $\Phi(t,\Omg_{n}) \sim \Omg_{n}$ is that once we define \begin{equation*}
\begin{split}
	I_{n}(t) = \int_{\Omg_{n}} \frac { \Phi_1(t,y) \Phi_2(t,y) }{|\Phi(t,y)|^5} \theta_0(y) \mathrm{d}y,
\end{split}
\end{equation*} we have \begin{equation*}
\begin{split}
	I_{n}(t) \ge a_{0} I_{n}(0)
\end{split}
\end{equation*} for some absolute constant $a_{0}$. The following claim gives the sharp bound on the ``almost invariant'' timescale for each bubble. 

\medskip

\noindent \textbf{Claim III}. For all $n\ge n_{0}$, we have \begin{equation*}
	\begin{split}
		\Phi(t,\Omg_{n}) \sim \Omg_{n}, \qquad \mbox{for} \qquad 0 \le t \le \min\{ T, \frac{c}{ M+\sum_{j=n_{0}}^{n-1} j^{-\alp} } \} =: T_{n}
	\end{split}
\end{equation*} with some constants $c,C>0$ independent of $n$. 

\medskip

\noindent This claim can be proved easily with an induction in $n$. In the base case $n=n_0$, we simply note that for $x \in \Phi(t,\Omg_{n_0})$, \begin{equation*}
	\begin{split}
		\left| \frac{u_j(t,x)}{x_j} \right| \le CM
	\end{split}
\end{equation*} from which the claim follows in this case. Assume that \textbf{Claim III} holds for all $n <n_{0}+k$ for some $k\ge1$. Note that using \textbf{Claim I} and the induction hypothesis, we have for $x \in \Phi(t, \Omg_{n_0+k})$ that \begin{equation*}
\begin{split}
	\left| \frac{u_j(t,x)}{x_j} \right| \le C\left(M + \sum_{\ell=0}^{k-1} I_{n_0+\ell} \right), \qquad 0 \le t \le T_{n_0+k-1}. 
\end{split}
\end{equation*} A simple application of Gronwall's inequality gives \textbf{Claim III}. \qedsymbol

We have proven that the $n$-th bubble remains almost invariant for $T_n$, which is bounded from below by \begin{equation*}
	\begin{split}
		T_{n} \ge \frac{c_0}{\sum_{j=n_{0}}^{n-1} j^{-\alp}  } \ge \frac{(1-\alp)c_0}{n^{1-\alp}} 
	\end{split}
\end{equation*} for all $n\ge N$ with some large $N$ depending only on $M,T$. Now, we observe that \begin{equation*}
\begin{split}
	\int_{0}^{T_{n}} I_{n}(t) \,\ud t \gtrsim T_{n}I_{n}(0) \gtrsim \frac{1}{n},
\end{split}
\end{equation*} with constants independent of $n$, recalling that $I_{n}(0) \gtrsim n^{-\alp}$. (We shall take $\alp$ close to $\frac{1}{2}$.) Hence, summation gives \begin{equation}\label{eq:c0}
\begin{split}
	\sum_{ k = \ell}^{n} I_{k}(0) T_{k} \ge c_{0}\left( \frac{1}{\ell} + \cdots + \frac{1}{n} \right) \ge \log\left( \frac{n}{\ell} \right)^{c_{0}}
\end{split}
\end{equation} for some absolute constant $c_{0}>0$, as long as $\ell>N$.

\medskip

\noindent \textbf{5. Norm inflation and conclusion the proof}. We are now in a position to complete the proof. For each $\ell>N$ and $n\gg\ell$ (so that $\log\left( \frac{n}{\ell} \right)^{c_0} \gg M$), we can bound for $x \in \Omg_{n}$ \begin{equation*}
	\begin{split}
		\log \frac{\widehat{\Phi}^n_{2}(T_{\ell})}{\widehat{x}^{n}_{2}} \le CM -10 \sum_{ k = \ell }^{n} I_{k}(0) T_{k}  \le\log \left( \frac{n}{\ell} \right)^{-c_{0}}.
	\end{split}
\end{equation*} In other words, we have the growth \begin{equation}\label{eq:growth}
\begin{split}
	\frac{\widehat{x}_2^{n}}{\widehat{\Phi}^{n}_{2}(T_{\ell})} \ge \left( \frac{n}{\ell} \right)^{c_{0}}. 
\end{split}
\end{equation} Now, we can write the solution in the form \begin{equation*}
\begin{split}
	\tht = \sum_{n=n_{0}}^{\infty} n^{-\alp} \tht^{(n)}, \qquad \tht^{(n)}(t,\Phi(t,x)) = \tht_{0,loc}^{(n)}(x), 
\end{split}
\end{equation*} so that the support of $\tht^{(n)}$ is disjoint from each other.  We now take $t = T_{\ell}$. Since $\tht^{(n)}(T_{\ell},\cdot) = 1$ in a region of area $\gtrsim 4^{-2n}$ and $\tht^{(n')} = 0$ for $n'\ne n$ in that region,  
with Hardy's inequality and \eqref{eq:growth}, we obtain that \begin{equation*}
\begin{split}
	\nrm{\tht^{(n)}(T_{\ell})}_{\dot{H}^{2}}^{2} \gtrsim   \left( \frac{n}{\ell} \right)^{2c_{0}}. 
\end{split}
\end{equation*} This estimate holds for all sufficiently large $n$. Then \begin{equation*}
\begin{split}
	\nrm{\tht(T_{\ell})}_{H^2}^2 \ge \sum_{n\ge n_0} n^{-2\alp} \nrm{\tht^{(n)}(T_{\ell})}_{\dot{H}^{2}}^{2} \gtrsim \ell^{-2c_{0}} \sum_{n\gg \ell} n^{2c_{0}-2\alp} . 
\end{split}
\end{equation*} In the last inequality, since $c_0>0$ is an absolute constant, and we could have chosen $\alp = \frac{1}{2}(1+c_0)$. This gives a contradiction to $\nrm{\tht(T_\ell)}_{H^2}<\infty$ since $\sum_{n \gg \ell} n^{-1+c_0} = \infty$. \end{proof}

\begin{remark}
	The nonexistence of the solution in $W^{1,\infty}$ follows directly. Repeating the above process with Lemma~\ref{key_lem2} instead of Lemma~\ref{key_lem}, we obtain \eqref{eq:growth}. Since we clearly have that
	\begin{equation*}
		\begin{split}
			\| \theta^{(n)}(T_{\ell}) \|_{\dot{W}^{1,\infty}} \gtrsim \left( \frac{n}{\ell} \right)^{c_{0}},
		\end{split}
	\end{equation*}
	it follows \begin{equation*}
		\begin{split}
			\nrm{\tht(T_{\ell})}_{W^{1,\infty}} \ge \sum_{n\ge n_0} n^{-\alp} \| \theta^{(n)}(T_{\ell}) \|_{\dot{W}^{1,\infty}} \gtrsim \ell^{-c_{0}} \sum_{n\gg \ell} n^{c_{0}-\alp} .
		\end{split}
	\end{equation*}
	Therefore, for the same $\alp = \frac{1}{2}(1+c_0)$, we complete the proof.

\end{remark}

\section{Norm inflation for smooth data}\label{sec:norm-inflation}

We establish Theorem \ref{thm:main} in this section, by proving a quantitative norm inflation result for data obtained by truncating the data used in the proof of Theorem \ref{thm:nonexist}. 

\begin{proposition}[Quantitative norm inflation]
	We consider the $C^\infty$--smooth initial data \begin{equation}\label{eq:initial-smooth}
		\begin{split} 
			\tht_0^{(N)} :=  \sum_{n=n_{0}}^{N}  n^{-\alp} \tht^{(n)}_{0, loc} 
		\end{split}
	\end{equation} where $\phi, \alp, n_0$ are the same as in \eqref{eq:nonexist-data}. Then, there exists $N_0 \ge 1$ depending only on $\phi, n_0$ such that for all $N\ge N_0$, the unique local in time $C^\infty$--solution $\tht^{(N)}$ to (SQG) with initial data $\tht_0^{(N)}$ exists on the time interval $[0,T^*]$ for some $0<T^*\le T_N$ and satisfies
\begin{equation}\label{eq:norm-inflation-smooth}
	\begin{split}
		\nrm{\tht_0^{(N)}}_{H^2 \cap W^{1,\infty}} \le \eps, \qquad  \sup_{t \in [0,T^*]}\nrm{\tht^{(N)}(t)}_{H^{2}}>M_N,
	\end{split}
\end{equation}  where \begin{equation}\label{eq:MNTN}
\begin{split}
	M_N := \frac{c_0}{2} \ln N, \qquad T_N := \frac{1}{M_N \ln M_N}
\end{split}
\end{equation} with $c_0>0$ from \eqref{eq:c0}. 
\end{proposition} 
\begin{proof}  We shall establish the proposition with a contradiction argument: let $0<T^*\le +\infty$ be the lifespan of the smooth solution associated with the initial data $\tht_0^{(N)}$ and assume that \begin{equation*}
		\begin{split}
			\nrm{\tht^{(N)}}_{L^\infty([0, \min\{ T^*, T_N \} ]; H^2)} \le M_N. 
		\end{split}
	\end{equation*} Under this contradiction hypothesis, we can actually prove that $T^*>T_N$, so that  \begin{equation}\label{eq:contradiction}
		\begin{split}
			\nrm{\tht^{(N)}}_{L^\infty([0,T_N ]; H^2)} \le M_N. 
		\end{split}
	\end{equation} This is simply because the $H^2$-norm gives a blow-up criterion. To illustrate this point, we estimate the $H^3$-norm of $\tht := \tht^{(N)}$ on $[0,T_N]$:\footnote{For simplicity, from now on we shall suppress from writing out the dependence of the solution $\tht$ in $N$.} from the equation for $\lap\tht$ \begin{equation*}
		\begin{split}
			\rd_t \lap\tht + u\cdot\nb \lap\tht + \lap u\cdot\nb\tht + 2\sum_{i=1,2} \rd_{i} u \cdot \nb \rd_i \tht = 0, 
		\end{split}
	\end{equation*} we estimate for $j = 1, 2$ \begin{equation*}
		\begin{split}
			\frac{1}{2} \frac{\ud}{\ud t} \nrm{\rd_{j}\lap\tht}_{L^2}^{2} \le C(\nrm{\nb u}_{L^\infty} + \nrm{\nb \tht}_{L^\infty}) \nrm{\rd_{j}\lap\tht}_{L^2}^{2} + C\nrm{\tht}_{H^2} \nrm{\tht}_{H^3}^2. 
		\end{split}
	\end{equation*} Here, we have used $L^4$ boundedness of the Riesz operator $\tht\mapsto u$ to bound \begin{equation*}
		\begin{split}
			\nrm{\nb^2 u}_{L^{4}} + \nrm{\nb^{2} \tht}_{L^{4}} \le C \nrm{\tht}_{H^2}^{\frac12}\nrm{\tht}_{H^{3}}^{\frac12}.
		\end{split}
	\end{equation*} Next, we use the logarithmic Sobolev inequality \begin{equation*}
		\begin{split}
			\nrm{\nb \tht}_{L^{\infty}} \le C \nrm{\tht}_{H^2} \log\left( 10 + \frac{ \nrm{\tht}_{H^3}}{\nrm{\tht}_{H^2}} \right)
		\end{split}
	\end{equation*} and \begin{equation*}
		\begin{split}
			\nrm{\nb u}_{L^{\infty}} \le C \nrm{u}_{H^2} \log\left( 10 + \frac{ \nrm{u}_{H^3}}{\nrm{u}_{H^2}} \right) \le C \nrm{\tht}_{H^2} \log\left( 10 + \frac{ \nrm{\tht}_{H^3}}{\nrm{\tht}_{H^2}} \right)
		\end{split}
	\end{equation*}(we have used the \textit{lower bound} $\nrm{u}_{H^2}\ge C\nrm{\tht}_{H^2}$). Lastly,  using $\nrm{\tht}_{H^2}\le M_N$, we  may deduce the \textit{a priori estimate} \begin{equation*}
		\begin{split}
			\frac{\ud}{\ud t} \nrm{\tht}_{H^3}^{2} \le C M_N\log\left( 10 +  \nrm{\tht}_{H^3}\right) \nrm{\tht}_{H^3}^{2}
		\end{split}
	\end{equation*} which shows that the $H^{3}$--norm of $\tht$ must remain finite up to $t=T_N$. Higher norms of $\tht$ can be similarly controlled, so that the solution $\tht$ remains $C^\infty$--smooth up to $t=T_N$. 

	\medskip
	
	\noindent In the following argument, $N_0 \gg n_0$ will be taken to be sufficiently large (but in a way depending only on a few absolute constants) whenever it becomes necessary. Recall that we are assuming $N\ge N_0$. The following argument is mainly a repetition of the proof of Theorem \ref{thm:nonexist} above. For convenience, let us fix \begin{equation}\label{eq:lN}
		\begin{split}
			\ell_N := M_N^3.
		\end{split}
	\end{equation} Then, note from the definition of $M_N$ in \eqref{eq:MNTN} that $n_0 \ll \ell_N \ll N$. Here and in the following, we write $A \ll B$ if $A/B \rightarrow 0$ as $N\to\infty$, where $A$ and $B$ are some positive expressions involving $N$. 
	
	Observe that the solution $\tht$ defined on $[0,T_N]$ satisfies the properties stated in Lemma \ref{lem:basic} and \textbf{Claim I, II, III} on the \textit{entire} time interval $[0,T_N]$ (by taking $N_0$ larger if necessary), simply because we have \begin{equation*}
		\begin{split}
			T_N \ll \frac{1}{M_N}
		\end{split}
	\end{equation*} from our choice of $T_N$ in \eqref{eq:MNTN}. As in the above, we write the solution in the form \begin{equation*}
	\begin{split}
		\tht = \sum_{n=n_{0}}^{N} n^{-\alp}\tht^{(n)}, \qquad \tht^{(n)}(t,\Phi(t,x)) = \tht_{0,loc}^{(n)}(x), 
	\end{split}
\end{equation*} and $\tht^{(n)}$ will be referred to as the $n$-th bubble. 
	Then, for any $\ell_N \le k \le N$, we have that the invariant timescale $T_k$ for the $k$-th bubble satisfies \begin{equation*}
		\begin{split}
			T_k \le T_N \qquad \mbox{and} \qquad  T_k \gtrsim \frac{1}{M_N + \sum_{j=n_0}^{k-1} j^{-\alp}} \gtrsim k^{\alp-1}.
		\end{split}
	\end{equation*} We have used that $\alp$ is close to $\frac12$. Now we consider the values of $n$ satisfying \begin{equation}\label{eq:n-cond}
	\begin{split}
		n \ge C \ell_N \exp(c_0^{-1}M_N)
	\end{split}
\end{equation} for a sufficiently large absolute constant $C>0$. Then, at $t = T_{\ell_N}$, we obtain similarly as before \begin{equation*}
\begin{split}
	\nrm{ \tht^{(n)}(T_{\ell_N})}_{\dot{H}^2}^2 \gtrsim \left( \frac{n}{\ell_N} \right)^{2c_0}
\end{split}
\end{equation*} whenever $n\le N$ satisfies \eqref{eq:n-cond}. Hence \begin{equation*}
\begin{split}
	\nrm{\tht(T_{\ell_N})}_{H^2}^2 \gtrsim \sum_{n = 1+\lfloor C \ell_N \exp(c_0^{-1}M_N) \rfloor }^{N} n^{-2\alp}  \left( \frac{n}{\ell_N} \right)^{2c_0} \gtrsim \ell_N^{-2c_0} N^{1-2\alp+2c_0} \gg M_N^{2}, 
\end{split}
\end{equation*} recalling the definitions of $M_N$ and $\ell_N$. We have used that $N \gg \ell_N \exp(c_0^{-1}M_N)$ to derive the last inequality. In particular, for all sufficiently large $N$, we obtain \begin{equation*}
\begin{split}
	\nrm{\tht(T_{\ell_N})}_{H^2} > M_N,
\end{split}
\end{equation*} which is a contradiction. This finishes the proof. 
\end{proof}

\section*{Acknowledgement}

\noindent IJ has been supported by the Samsung Science and Technology Foundation under Project Number SSTF-BA2002-04 and the New Faculty Startup Fund from Seoul National University.

% ----------------------------------------------------------------
\bibliographystyle{amsplain}

% ----------------------------------------------------------------

\end{document}